\documentclass[11pt]{article}                    
\usepackage{graphicx,color,amsthm}
\usepackage{amssymb}
\usepackage{amsmath}
\usepackage{tikz}
\usepackage{float}
\usetikzlibrary{arrows}

\DeclareSymbolFont{largesymbol}{OMX}{yhex}{m}{n}
\DeclareMathAccent{\Widehat}{\mathord}{largesymbol}{"62}

\newcommand{\bm}{\boldsymbol}

\theoremstyle{plain}
\newtheorem{theorem}{Theorem}
\newtheorem{lemma}{Lemma}

\begin{document}

\title{On the Asymptotic Behavior of the Kernel Function in the Generalized Langevin Equation: A One-dimensional lattice model
\thanks{This research was supported by NSF under grant DMS-1522617 and DMS-1619661. }
}


\author{Weiqi Chu    \\
The Pennsylvania State University \\   {wzc122@psu.edu}   \and
        Xiantao Li \\ 
        The Pennsylvania State University \\
        {xxl12@psu.edu}
}


%
%
\maketitle

\begin{abstract}
We present some estimates for the memory kernel function in the generalized Langevin equation, derived using the Mori-Zwanzig
formalism from a one-dimensional lattice model, in which the particles interactions are 
through nearest and second nearest neighbors. The kernel function can be
explicitly expressed in a matrix form. The analysis focuses on the decay properties, both spatially and temporally, revealing
a power-law behavior in both cases. The dependence on the level of coarse-graining is also studied.
\end{abstract}

\section{Introduction}\label{intro}
The molecular dynamics models, often written in the form of a Hamiltonian system or a Langevin dynamics model, 
are playing increasingly important roles in understanding the structural properties of macromolecules and material systems \cite{Leach01}. 
There is also enormous interest in reducing the enormous degrees of freedom associated with such models to a few coarse-grained variables that are able to describe the overall collective properties. The Mori-Zwanzig (MZ) formalism \cite{ChHaKu02,ChHaKu00,Mori65,Zwanzig73} has recently re-emerged as a systematic methodology for such dimension-reduction purposes. In principle, the MZ formalism yields an {\it{exact}} equation, and with appropriate approximations, the equation can be written as a generalized Langevin equation (GLE).
A distinct feature of the GLEs is an integral that incorporates the history-dependence, and a random noise term, which satisfies the second fluctuation-dissipation theorem \cite{Kubo66}. 

The GLEs have opened up a new paradigm in molecular modeling. Numerous attempts have been made to derive the GLEs
from the full molecular models, and the models have been applied to solid-gas interface \cite{AdDo76,AdDo74}, protein dynamics \cite{berkowitz1981memory,CuCe02,lange2006collective,oliva2000generalized,stepanova2007dynamics}, crystalline solids \cite{kauzlaric2011bottom,Li2009c,Li14}, etc. In general, little is known about the mathematical properties of the memory kernel. A very interesting experimental work \cite{min2005observation,Kou_Xie_PRL_2004} has suggested that part of a protein dynamics can be described by generalize Langevin equations, in which the kernel function exhibits a decay in time that is proportional to $t^{-0.51}.$ 

On the other hand, it is of great practical interest to introduce approximations to the kernel function, to obtain an efficient method for solving the generalize Langevin equations. For instance, a brute force solution method would have to store the solution for all previous steps, and then at every time step, the integral has to be evaluated, which further adds up to the overall computational cost. It is important to understand properties of the kernel functions so that appropriate approximations can be introduced.  

Motivated by these observations, we present several analytical results, with particular emphasis on the following questions:
\begin{enumerate}
\item How does the memory function depend on the choice of the coarse-grained variables?
\item How would the kernel function behave at different levels of coarse-graining? 
\item How does the memory function decay in space and time?
\end{enumerate}

To be able to quantify these aspects, we consider a one-dimensional lattice model, in which the particles are connected by
linear springs. In this case, the generalize Langevin equations can be derived and the kernel functions can be explicitly expressed as matrices or operator functions. The coarse-grained variables can be defined based on local averages within equally partitioned blocks. There are also different
ways to define weights of the averaging, ranging from piecewise constant, overlapping hat functions, or piecewise quadratic functions. These constructions have been motivated by the finite element shape functions. 

To address the first two issues, we study the memory function at $t=0,$ by varying the block size, denoted by $M$. $M$ would represent the coarse scale that we are upscaling to, or the level of coarse-graining. This is done for both piecewise constant and piecewise linear weighting functions. We found that the memory function in the second case decays faster as $M\to+\infty.$ For the last issue, we hold $M$ fixed, and analyze the spatial and temporal decay of the memory function. We found that the decay rate in time is $1/2,$ which seems to be close to the experimental observation \cite{min2005observation,Kou_Xie_PRL_2004}. It is also interesting that the generalized Langevin equation derived by Adellman and Doll \cite{AdDo76} exhibits the same decay rate.

The remaining part of the paper is organized as follows. In section 2, we present the one-dimensional model, along with the generalized Langevin equation. The kernel function will be represented in terms of its Fourier transform. In section 3, we present the analysis of several properties of the kernel function, along with the verification via numerical tests.

\section{A one-dimensional model and the generalized Langevin equation}
\label{sec1:model}
\subsection{The one-dimensional lattice dynamics model}

We focus our analysis on a one-dimensional lattice model that has served as an excellent test model to study heat conduction, fracture, analysis of modeling errors, etc., e.g., in \cite{lepri2003thermal,lepri1998anomalous,DobsonLuskin:2009b,FuTh78,LiMing:2013,MingYang:2009} and many related works. This model consists of a chain of atoms, with the position of $j$th atom denoted by $r_j$, $j\in \mathbb{Z}$. The equations of motion are expressed as
\begin{equation}
m\ddot{r}_{j} = - \frac{\partial V({r})}{\partial r_{j}}, \quad j\in \mathbb{Z}.
\label{eq:newton}\end{equation}
Here we have used $V({r})$ for the atomic interactions among the atoms.

As a specific example, we assume that particles have direct interactions among first and second neighbors. The potential energy can be written as
\begin{equation}
V = \frac{1}{2} \sum_{j\in \mathbb{Z}} \varphi(r_{j+1}-r_{j}) + \varphi(r_{j+2}-r_{j}),
\label{eq:potential}
\end{equation}
where $\varphi(r)$ is the potential energy associated with two interacting atoms of distance $r$. Atoms typically move around their equilibrium positions, denoted by $R_j=ja_0$, and $a_0$ is the equilibrium atom spacing (aka lattice constant). We define $u_{j} = r_{j} - ja_0$ as the displacement. We also set $m=1$ for simplicity. 

By expanding \eqref{eq:newton} around equilibrium positions $R_j$ and neglecting high order terms, one can have
\begin{equation}
\ddot{{u}} = -{\cal A} {u},
\label{eq:linear}
\end{equation}
where ${u}$ represents the displacement of all the atoms and $\cal{A}$ is a finite difference operator, defined as
\begin{equation}
 \big({\cal A}{u}\big)_j\overset{\rm def}{=} -\kappa_{2} u_{j-2} - \kappa_{1} u_{j-1} + \kappa_{0} u_j - \kappa_{1} u_{j+1} - \kappa_{2} u_{j+2},
 \label{eq: opA}
\end{equation}
where $\kappa_{1} = \frac{1}{2}\varphi''(a_0)$, $\kappa_{2}=\frac{1}{2}\varphi''(2a_0)$ and $\kappa_{0}=2(\kappa_{1}+\kappa_{2})$. This linear approximation of the atomic forces is known as the harmonic approximation \cite{ashcroft1976solid}, and it is a useful routine to obtain elastic parameters, the phonon dispersion relations, stability conditions, etc. 

We assume that $\cal A$ is positive semidefinite in appropriate function spaces, which is pertinent to the stability of the lattice structure \cite{weinan2007cauchy}. For the model \eqref{eq: opA}, this amounts to the following phonon stability conditions:
\begin{equation}
\kappa_1 > \text{ and } \kappa_{1} + 4\kappa_{2} >0.
\label{eq:stability}
\end{equation}
In fact, it is not difficult to verify that $\cal A$ is a discrete Laplacian operator. 

\subsection{The generalized Langevin equation as a coarse-grained model}

We now discuss the reduction of the full dynamics model \eqref{eq:linear}, and the emergence of the generalized Langevin equation. Known as {\it coarse-graining}, such a procedure is central to modern molecular modeling. The first step in this effort is to map the atom position and velocity to a set of coarse-grained variables. For this purpose, we will introduce some averaging operators. We let the one-dimensional infinite lattice be $\mathbb{L}$; $\mathbb{L}=a_0 \mathbb{Z}.$ We partition the entire lattice into blocks, each of which contains $M$ atoms and use $\mathbb{L}_{CG}=M\mathbb{L}$ to represent the coarse-grained space. The coarse-grained variables will be defined with appropriate local spatial averaging.

A natural idea is to extract the usual average from a block as the coarse-grained variable. In this case, the weights are called piecewise constant weighting functions: $1/M$ for atoms within the block, and zero otherwise, as shown in Fig. \ref{fig: const}.

\begin{figure}[htp]
\begin{center}
\fbox{
\begin{tikzpicture}
\draw[-] (-3,0) -- (5,0) ; 
\draw[-] (0,1) -- (2,1) ; 
\foreach \x in {0,2} 
\draw[-,dashed] (\x,0) -- (\x,1) node[above] {};
\draw[shift={(-2,0)},color=black] (0pt,0pt) -- (0pt,-4pt) node[below] {$-M$};
\draw[shift={(0,0)},color=black] (0pt,0pt) -- (0pt,-4pt) node[below] {$0$};
\draw[shift={(2,0)},color=black] (0pt,0pt) -- (0pt,-4pt) node[below] {$M$};
\draw[shift={(4,0)},color=black] (0pt,0pt) -- (0pt,-4pt) node[below] {$2M$};

\end{tikzpicture}}
\centering
\caption{The constant weights for the averaging over one block.}
\label{fig: const}
\end{center}
\end{figure}
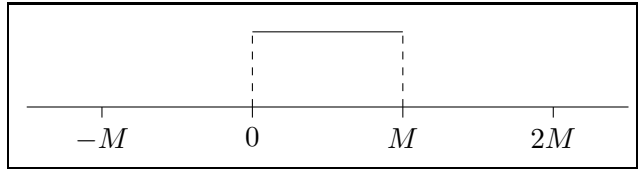

On the other hand, one may choose non-uniform weighting functions, e.g., as shown in Fig. \ref{fig: linear}, the weights are overlapping hat functions and the averaging is over two adjacent blocks. We call it piecewise linear averaging. This is motivated by the restriction operator in multigrid methods \cite{brenner2007mathematical}.


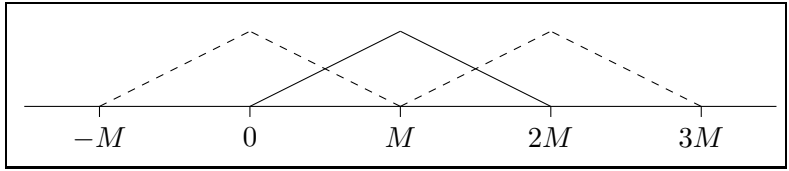
\begin{figure}[htp]
\begin{center}
\fbox{

\begin{tikzpicture}

\draw[-] (-3,0) -- (7,0) ; 

\draw[-,dashed] (-2,0) -- (0,1) node[above] {}; \draw[-,dashed] (2,0) -- (0,1) ;
\draw[-,dashed] (2,0) -- (4,1) node[above] {}; \draw[-,dashed] (6,0) -- (4,1) ;
\draw[-] (0,0) -- (2,1) ; \draw[-] (4,0) -- (2,1) ;

\draw[shift={(-2,0)},color=black] (0pt,0pt) -- (0pt,-4pt) node[below] {$-M$};
\draw[shift={(0,0)},color=black] (0pt,0pt) -- (0pt,-4pt) node[below] {$0$};
\draw[shift={(2,0)},color=black] (0pt,0pt) -- (0pt,-4pt) node[below] {$M$};
\draw[shift={(4,0)},color=black] (0pt,0pt) -- (0pt,-4pt) node[below] {$2M$};
\draw[shift={(6,0)},color=black] (0pt,0pt) -- (0pt,-4pt) node[below] {$3M$};

\end{tikzpicture}}
\caption{The piecewise linear weights for the averaging over blocks.}
\label{fig: linear}
\end{center}
\end{figure}

Let $X=\ell^2(\mathbb{L})$ be the space for the atomic displacement. Clearly, the weighting functions span a subspace, denoted by $Y$, and $Y\subset X$. Essentially, there also exists a bijection between $Y$ and $\ell^2(\mathbb{L}_{\text{CG}})$ and the averaging is a mapping from $X$ to $Y$ and defines the coarse-grained variables. 

To start with, we define a surjective operator $\Phi:X\rightarrow Y$, the matrix representation of which has columns that are a set of orthonormal basis of $Y$. The coarse-grained variables are defined in a compact form,
\begin{equation}\begin{aligned}
{q}=\Phi^{\intercal}{u}, \\
{p} = \Phi^{\intercal}{v},
\end{aligned}\end{equation}
where $\Phi^\intercal$ is the adjoint operator of $\Phi$, and ${v}=\dot{{u}}$ can be interpreted as the velocity. 

It is also useful to consider the orthogonal complement $Y^\perp$, which contains the `extra' degrees of freedom. Let $\Psi$ be a surjective from $X$ to $Y^\perp$. When $\Psi$ is seen as an infinite dimensional matrix, columns of $\Psi$ consist of a set of orthogonal basis of $Y^{\perp}$. We have a unique decomposition in the form of
\begin{equation}\begin{aligned}
{u} = \Phi {q}+\Psi {\xi}, \\
{v} = \Phi {p} + \Psi {\eta}. 
\end{aligned}
\end{equation}

Once the coarse-grained variables are selected, the Mori-Zwanzig formalism can be invoked to derive a reduced model, in the form of generalized Langevin equations, 
\begin{equation}\label{eq: GLE}
\left\{
\begin{aligned}
  \dot{{q}}=& ~{p},\\
  \dot{{p}}= &-{\cal A}{q} - \int_0^t \Theta(t-\tau) {p}(\tau) d\tau + {R}(t). 
\end{aligned}\right.
\end{equation}
In particular, the kernel function is given by \cite{lee2017semi},
\begin{equation}
\Theta(t) = \Phi^{\intercal}{\cal A}\Psi\cos(\Omega t)\Omega^{-2}\Psi^{\intercal}{\cal A}\Phi,
\label{eq:kernel}
\end{equation}
where $\Omega^{2}=\Psi^{\intercal}{\cal A}\Psi$ and the cosine function is defined by the Taylor series, 
\begin{equation}
\cos(\Omega t)=I - \frac{(\Omega t)^2}{2!} +\frac{(\Omega t)^4}{4!} - \frac{(\Omega t)^6}{6!} + \cdots.
\end{equation}
Furthermore, ${R}(t)$ is a stationary Gaussian process. Implicit in the equation \eqref{eq: GLE}, there is an important relation between the frictional kernel $\Theta$ and the random force, known as the second fluctuation-dissipation theorem \cite{Kubo66}, \begin{equation}\label{eq: 2FDT}
 \Big\langle {R}(t) {R}(t)^\intercal \Big\rangle= k_B T \Theta(t-t'),
\end{equation}
where $T$ is the temperature and $k_B$ is the Boltzmann constant. This condition is necessary for the system to equilibrate to its correct probability distribution.

For the detailed derivations, see \cite{LiXian2014,Li14,Li2009c}. Similar derivations can be found under other settings \cite{AdDo76,AdDo74,lee2017semi,LiE07,LiE06}.

\noindent{\bf Remark:} It is worthwhile to point out that we have selected the infinite lattice dynamics as the full model. This is because for finite systems, the kernel function often exhibits a `recurrence' phenomenon. Namely, the value of the kernel function first decays, and then after some time, the values go back up. The recurrence time increases as the system size increases. But it is difficult to have a decay property within a finite system.

\medskip

As alluded to in the introduction, it is useful to understand how the kernel function decays in space and time, and also the dependence on the level of coarse-graining, i.e., $M$. Our next task is to estimate the asymptotic behaviors of $\Theta(t)$ from different perspectives.

From the definitions, operators ${\cal A}$ and $\Phi$ are Toeplitz forms \cite{Ciarlet1970} and exhibit translational symmetry. Use Fourier transform can simplify the expressions a lot. For this purpose, we see $\cal A$ as an operator from $\ell^2 \left(\mathbb{L}_{\text{CG}}\right)^M$ to $\ell^2(\mathbb{L}_{\text{CG}})^M$ and look at its matrix representation form. We split $\cal A$ into block matrices of size $M$-by-$M$. Since atoms only interact within the same block or neighboring blocks, overall, $\cal A$ has a block tridiagonal structure. In fact, $\cal A$ is a Toeplitz form generated by the sequence,
\begin{equation}
\left( \cdots,0,0, A_{-1}^{\intercal},A_0^{\intercal},A_1^{\intercal},0,0,\cdots \right) ^{\intercal},
\end{equation} 
where $A_{\text{-}1},A_{0}$ and $A_{1}$ belong to $\mathbb{R}^{M\times M}$.

As a result, $\cal A$ can be formally expressed in the matrix form
\[\cal{A}=\left[
\begin{array}{ccccccc}
 \ddots & \ddots 	& \ddots &  & & \\
\cdots &A_1 & A_0 & A_{\text{-1}} & 0 &0 &\cdots\\
\cdots &0&A_1 & A_0 &  A_{\text{-1}} &0&\cdots\\
\cdots &0 &0 & A_1 & A_0 & A_{\text{-1}} &\cdots\\
& & & & \ddots& \ddots&\ddots \\
\end{array}\right].
\]
The diagonal block, denoted by $A_0$, is a band matrix with bandwidth 5, containing the force constants,
\[
A_0=\left[
\begin{array}{ccccccccc}
\kappa_{0} & \text{-}\kappa_{1} & \text{-}\kappa_{2} &  0 & 0 & 0  &0  & \cdots & 0\\
 \text{-}\kappa_{1} & \kappa_{0} & \text{-}\kappa_{1} & \text{-}\kappa_{2} &  0 & 0  &0    & \cdots  & 0\\
\text{-}\kappa_{2} & \text{-}\kappa_{1} & \kappa_{0} & \text{-}\kappa_{1} & \text{-}\kappa_{2} &  0 & 0  &\cdots  & 0\\
 0 & \text{-}\kappa_{2} & \text{-}\kappa_{1} & \kappa_{0} & \text{-}\kappa_{1} & \text{-}\kappa_{2} &  0 & \cdots  & 0\\
0 & 0 & \text{-}\kappa_{2} &\text{-}\kappa_{1} &\kappa_{0}&\text{-}\kappa_{1} &\text{-}\kappa_{2} &\cdots & 0\\
\vdots &\vdots &  &\ddots &\ddots &\ddots &\ddots &\ddots &\vdots\\
0 & 0 & \cdots &\cdots & \text{-}\kappa_{2} &  \text{-}\kappa_{1}  & \kappa_{0} & \text{-}\kappa_{1} &\text{-}\kappa_{2}\\
0 & 0 & \cdots &\cdots & 0 &  \text{-}\kappa_{2}  & \text{-}\kappa_{1} & \kappa_{0} & \text{-}\kappa_{1}\\
0 & 0 & \cdots &\cdots & 0 & 0 & \text{-}\kappa_{2} & \text{-}\kappa_{1} &\kappa_{0}\\
\end{array}\right].
\]
The off-diagonal block, denoted by $A_{1}$, only has three non-zero entries,
\begin{equation*}\begin{aligned}
&A_{1}(1,M) = -\kappa_{1},\\
&A_{1}(1,M -1) = A_1(2,M) = -\kappa_{2},
\end{aligned}\end{equation*}
and $A_{\text{-1}}={A}^{\intercal }_1$.

In accordance with the partition of $\cal A$, one can always regard ${f}\in \ell^2(\mathbb{L})$ as a function in $\ell^2(\mathbb{L}_{\text{CG}})^M$, denoted by ${\bm{f}}$. We break ${f}$ into blocks, with one block containing $M$ elements and assemble blocks into an array $\bm f$, with $\bm{f}_n\in \mathbb{R}^M$ being the $n$th component of $\bm{f}$, i.e.,
\begin{equation}
\bm{f}_n = (f_{nM},f_{nM+1},\cdots,f_{nM+M-1})^{\intercal}.
\end{equation}

Now we can define the Fourier transform of $\bm{f}$ at the level of blocks,
\begin{equation}
\Widehat{\bm f}(\xi) = \sum_{n}e^{-in\xi} \bm f_{n}, \quad ~\xi \in (0,2\pi],
\end{equation}
with inverse given by,
\begin{equation}
 \bm f_{n} = \frac{1}{2\pi} \int_0^{2\pi} \Widehat{\bm f}(\xi)e^{in\xi} d\xi.
\end{equation}

Similarly, one can also define the Fourier transform of Toeplitz forms in terms of its matrix representation. Suppose $G$ is a Toeplitz form, whose columns of blocks are generated by a column: $$\bm G=(\cdots, G_{-2}^{\intercal},G_{-1}^{\intercal},G_{0}^{\intercal},G_{1}^{\intercal},G_{2}^{\intercal},\cdots)^{\intercal}.$$
The Fourier transform of the operator $G$ is defined as
\begin{equation}
 \Widehat{G}(\xi) = \sum_{n}e^{-in\xi}G_{n},\quad ~\xi \in (0,2\pi].
\label{eq:Fourierform}
\end{equation}

Based on the block form, the Fourier transformation of $\cal A$ is simply defined as
\begin{equation}
\Widehat{\cal A}(\xi)=A_{0}+e^{-i\xi}A_{1}+e^{i\xi}A_{1}^{\intercal}, \quad ~\xi \in (0,2\pi].
\end{equation}

From the perspective in solid state physics, grouping atoms in blocks is equivalent to building a complex lattice with $M$ atoms in one primitive cell. The Fourier domain $(0, 2\pi]$ is equivalent to the first Brillouin zone $(-\pi,\pi]$.

In terms of the Fourier transforms, it is not difficult to establish
\begin{equation}
 \Widehat{{\cal A}\bm{f}}= \Widehat{\cal A}\Widehat{\bm{f}}\quad \text{and} \quad \Widehat{GH}=\Widehat{G}\Widehat{H}
 \end{equation}
where $G$ and $H$ have consistent domains and ranges.

Notice that the kernel function \eqref{eq:kernel} is represents an operator from $\ell^2(\mathbb{L}_{\text{CG}})$ to $\ell^2(\mathbb{L}_{\text{CG}})$. Therefore, for analysis, it is necessary to convert the system into a finite-dimensional one. The Fourier transformation in terms of blocks
is a useful tool to achieve this goal. 

Applying Fourier transform to the kernel function in \eqref{eq:kernel}. With direct computation, we find that the $(k,\ell)$-th entry of $\Theta(t)$ can be written as
\begin{equation}
\Theta_{k\ell}(t) = \frac{1}{2\pi}\int_{0}^{2\pi}e^{i(\ell-k)\xi}\Widehat{\Phi}^{*}\Widehat{\cal A}\Widehat{\Psi}\cos(\Widehat{\Omega}t)\left(\Widehat{\Psi}^{*}\Widehat{\cal A}\Widehat{\Psi}\right)^{-1}\Widehat{\Psi}^{*}\Widehat{\cal A}\Widehat{\Phi}~d\xi,
\label{eq:lakernel}
\end{equation}
where $\Widehat{\Phi}, \Widehat{\Psi}$ and $\Widehat{\Omega}$ are defined based on \eqref{eq:Fourierform}. Our asymptotic analysis will be established based on the representation of the kernel function in \eqref{eq:lakernel}.

\section{Asymptotic behavior of kernel functions}
\label{asymbeha}
The kernel function depends on the choice of the CG variables, which can be defined based on the averaging methods within the blocks. The kernel function also depends on the block size $M$, which shows the level of coarse-graining. 

Here, the dependence of the block size is discussed for two types of averaging, piecewise constant averaging and piecewise linear averaging, which are two most common and intuitive ways in the finite element methods \cite{brenner2007mathematical}. In this section, we first study the asymptotic behavior of $\Theta_{0,0}(0)$ in the kernel function as the block size $M$ goes to infinity.

Meanwhile, the coarse-grained variables are defined based on different weighting functions. In its matrix form, every column of $\Phi$ represents a weighting function centered at one block. The correlation between two blocks is expected to become smaller as the distance of the two blocks increases. This corresponds to the entry $\Theta_{0,J}(0)$. We will analyze how the kernel function decays as $J$ goes to infinity.

Finally, we also analyze the dependence of the kernel functions on time $t$ when other parameters are fixed, such as the weighting functions and the block size. $\Theta_{0,0}(t)$ also decays over a long time period and we estimate its decay rate.

\subsection{The dependence on the block size $M$}

\subsubsection{Piecewise constant weighting function}
Recall that $\Phi$ is a Toeplitz form, arising from the piecewise constant averaging. $\Phi$ can also be seen as an infinite dimensional block diagonal matrix, with diagonal block equal to
\begin{equation}
Q_{1}=\frac{1}{\sqrt{M}}\left(1,1,\cdots,1\right)^{\intercal},
\end{equation}
and $Q_{1} \in \mathbb{R}^{M}.$ It is not difficult to find $Q_{2}\in \mathbb{R}^{M\times (M\text{-}1)}$ whose column vectors are a set of orthonormal basis in the complementary subspace spanned by $Q_{1}$, i.e., $[Q_1 \; Q_2]\in \mathbb{R}^{M\times M}$ is an orthogonal matrix. $\Psi$ can be defined as a block diagonal form with diagonal block equal to $Q_{2}$. 

According to \eqref{eq:lakernel}, in terms of $Q_1$ and $Q_2$, the diagonal entry of $\Theta$ can be written as
\begin{equation}
\Theta_{0,0}(0) = \frac{1}{2\pi}\int_{0}^{2\pi} Q_{1}^{\intercal}\Widehat{\cal A}Q_{2}\left(Q_{2}^{\intercal}\Widehat{\cal A}Q_{2}\right)^{\!-1}\!\!Q_{2}^{\intercal}\Widehat{\cal A}Q_{1}~d\xi.
\label{eq:constest}
\end{equation}

On the right-hand side of \eqref{eq:constest}, $Q_1$ comes from the weighting functions and is known to us. Provided the orthogonality of $Q_{1}$ and $Q_{2}$, the form of $Q_2$ is not unique. So it is necessary to show that $\Theta_{0,0}(0)$ in \eqref{eq:constest} does not depend on the specific choice of $Q_2$. Another difficulty is that the expression contains the inverse of $Q_{2}^{\intercal}\Widehat{\cal A}Q_{2}$ which is not easy to compute in general.

Fortunately, we have found that $Q_1$ and $Q_2$ have properties that allow us to resolve these issues. We begin with the following formula.

\begin{lemma} 
[Inverses of $2$-by-$2$ block matrices \cite{LU2002119}]
Suppose $C$ and $D$ are both 2-by-2 block matrices. $C_{11}$ and $D_{11}$ belong to $\mathbb{C}^{n_{1}\times n_{1}}$ and $C_{22}$ and $D_{22}$ belong to $\mathbb{C}^{n_{2}\times n_{2}}$. If $CD=I$ and $D_{11}, D_{22}$ are invertible, then the Shur Complement $D_{22}-D_{21}D_{11}^{-1}D_{12}$ is invertible and 
\begin{equation}
C_{22} = \left(D_{22}\!-\!D_{21}D_{11}^{-1}D_{12}\right)^{-1}.
\label{eq:inverse}
\end{equation}
\label{lm1}
\end{lemma}

From direct calculations, the following lemma can be established.
\begin{lemma}
\label{lemma:eigenA}
Let ($\lambda_{j},{v}_{j})$ be an eigenpair of $\Widehat{\cal A}$, namely, $\lambda_j$ is a scalar, ${v}_j\in \mathbb{C}^M$, and $\Widehat{\cal A}{v}_{j}=\lambda_{j}{v}_{j}$. Then 
\begin{align}
&\lambda_{j} = 2(\kappa_{1}\!+\!\kappa_{2}) -2\kappa_{1}\cos(\xi_{j}') -2\kappa_{2}\cos(2\xi_{j}'), \\
& {v}_{j} = \frac{1}{\sqrt{M}}\left[\!\!\!\!
\begin{tabular}{c}
1 \\ $e^{i\xi_{j}'}$ \\ $e^{2i\xi'_{j}}$ \\$ \vdots $\\ $e^{i(M-1)\xi_{j}'}$
\end{tabular}\!\!\!\!\right], \quad \xi_{j}' = \frac{-\xi+2j\pi}{M}, \quad j=0,1,\!\cdots\!, M\!-\!1.
\end{align}
\end{lemma}
From Lemma \ref{lemma:eigenA} it can be seen that the stability condition \eqref{eq:stability} guarantees that $\Widehat{\cal A}$ is positive semidefinite for any $\xi$ in $(0,2\pi]$.

\begin{theorem} If the coarse-grained variables are defined based on piecewise constant averaging operator, then the following estimate holds,
\begin{equation}
0 \le \Theta_{0,0}(0) \leq \frac{2\kappa_{1}+4\kappa_{2}}{M}.
\end{equation}
\label{thm:1}
\end{theorem}

\begin{proof}
For $\xi \ne 2\pi$, $\Widehat{\cal A}$ is invertible and from the orthogonality of $Q_{1}$ and $Q_{2}$,
\begin{equation}
\left(\left[
\begin{array}{c}
Q_{1}^\intercal \\
Q_{2}^\intercal 
\end{array}
\right]\Widehat{\cal A} \; \big[Q_1\; Q_2\big]\right)^ {-1} \!\!=\;
 \left[
\begin{array}{c}
Q_{1}^\intercal \\
Q_{2}^\intercal 
\end{array}
\right]\Widehat{\cal A}^{\text{-}1} \big[Q_1\; Q_2\big].
\end{equation}
From Lemma \ref{lemma:eigenA}, $\Widehat{\cal A}$ has a zero eigenvalue only when $\xi=2\pi,$ and its corresponding eigenvector is equal to $\Widehat{\Phi}$ in the form of piecewise constant averaging. So $\Widehat{\Phi}^{*}\Widehat{\cal A}\Widehat{\Phi}$ is invertible for $\xi\in (0,2\pi)$. Since $\Widehat{\Phi}$ is perpendicular to $\Widehat{\Psi}$, $\Widehat{\Psi}^* \Widehat{\cal A} \Widehat{\Psi}$ is non-singular for $\xi \in (0,2\pi)$.

Applying Lemma \ref{lm1} to the above matrices, we arrive at the following equality,
\begin{equation}
\begin{array}{cc}
\Big(Q_{2}^{\intercal}\Widehat{\cal A}Q_{2}\Big)^{-1} = Q_2^{\intercal}\Widehat{\cal A}^{\text{-}1} Q_{2} -Q_{2}^{\intercal}\Widehat{\cal A}^{\text{-}1} Q_{1} \left(Q_{1}^{\intercal}\Widehat{\cal A}^{\text{-}1} Q_{1} \right)^{-1} Q_{1}^{\intercal}\Widehat{\cal A}^{\text{-}1} Q_{2}
\end{array}.
\label{eq:Theta}
\end{equation}

Substituting \eqref{eq:Theta} into \eqref{eq:constest}, and using the facts that $Q_{1}Q_{1}^{\intercal}+Q_{2}Q_{2}^{\intercal}=I$ and $\left(Q_{1}^{\intercal}\Widehat{\cal A}^{\text{-}1} Q_{1} \right)^{-1}$ is a scalar, one obtains,
\begin{equation}
\Theta_{0,0}(0) = \frac{1}{2\pi}\int_{0}^{2\pi}Q_{1}^{\intercal}\Widehat{\cal A}Q_{1} - (Q_{1}^{\intercal}\Widehat{\cal A}^{-1}Q_{1})^{-1} \text{ d}\xi.
\label{eq:est1}\end{equation}

In light of \eqref{eq:lakernel} and the eigen decomposition in Lemma \ref{lemma:eigenA}, we know that $\Theta_{0,0}(0)$, $Q_1^{\intercal}\Widehat{\cal A}Q_1$ and $(Q_{1}^{\intercal}\Widehat{\cal A}^{-1}Q_{1})^{-1}$ are real and non-negative. Thus, it is sufficient to use the integral of the first term to bound $\Theta_{0,0}(0)$,
\begin{equation}\begin{aligned}
\frac{1}{2\pi}\int_{0}^{2\pi}Q_{1}^{\intercal}\Widehat{\cal A}Q_{1}\text{ d}\xi &= \frac{1}{2\pi}\int_{0}^{2\pi}Q_{1}^{\intercal}\left( A_{0}+e^{i\xi}A_{1}+e^{-i\xi}A_{1}^{\intercal}\right)Q_{1}\text{ d}\xi \\
&= Q_{1}^{\intercal} A_{0} Q_{1} = \frac{2\kappa_{1}+4\kappa_{2}}{M}. \\
\end{aligned}\label{eq:est2}\end{equation}
So $\Theta_{0,0}(0) $ can be bounded by,
\begin{equation}
0 \leq \Theta_{0,0}(0) \leq \frac{1}{2\pi}\int_{0}^{2\pi}Q_{1}^{\intercal}\Widehat{\cal A}Q_{1}\text{ d}\xi = \frac{2\kappa_{1}+4\kappa_{2}}{M}.
\end{equation}
\end{proof}
\qed

\noindent{\bf Remark:} The second fluctuation-dissipation theorem \eqref{eq: 2FDT} indicates that $\Theta_{0,0}(0)$ is proportional to the variance of the random noise. This estimate shows that the magnitude of the noise is inverse proportional to the level of coarse-graining. 

Due to the second fluctuation-dissipation theorem \eqref{eq: 2FDT} and Cauchy-Schwarz inequality, we have
\begin{equation}\begin{aligned}
\Theta_{I,J}^{2}(0) &= \left(k_{B}T\right)^{-2} \big\langle R_{I}(0)R_{J}(0)\big\rangle ^{2} \\
&\le \left(k_{B}T\right)^{-2} \big\langle R_{I}(0)R_{I}(0)\big\rangle \big\langle R_{J}(0)R_{J}(0)\big\rangle \\
&=\Theta_{I,I}(0)\Theta_{J,J}(0). \\
\end{aligned}\end{equation}
Since $\Theta$ is a Teoplitz form, we have $\Theta_{I,I}=\Theta_{J,J}=\Theta_{0,0}$. This implies off-diagonal entries can also be bounded by the diagonal entry and our estimate,
\begin{equation}
|\Theta_{I,J}(0)| \le |\Theta_{0,0}(0)| \le {\cal O}(M^{-1}).
\end{equation} 

\

Here we present a numerical experiment for the 1-d lattice model in Figure \ref{fig:blocksize-const} to verify this estimate. The force constants are obtained from a Morse potential, with $\kappa_{1} = 12.2676$ and $\kappa_{2} = 3.0628$. We generate a large enough system ($N=2^{13}$) with periodic boundary conditions to imitate the 1-d infinite chain. When the number of atoms is large enough, the boundary conditions will have little effect on interior atoms within a relatively short time period. So it is reasonable to use this finitely large system to approximate the exact model. From Figure \ref{fig:blocksize-const}, we observe that the estimate is quite sharp.

\begin{figure}[ht]
\begin{center}
 \includegraphics[width=0.5\textwidth,height=0.4\textwidth]{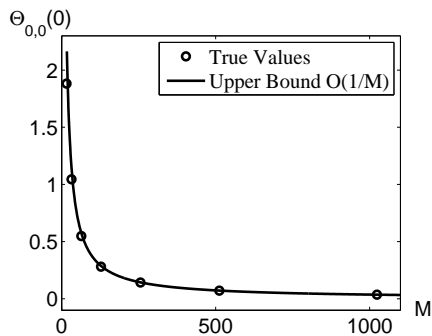}
 \caption{This figure shows the dependence of the kernel function on the block size for the piecewise constant averaging operator. We fix the total number of the atoms $N=2^{13}$ in the simulation and vary the block size $M=2^{4},2^{5},\cdots,2^{10}$, which represents the scale of averaging. We assemble $M$ atoms as a block. The diagonal entry of the kernel function $\Theta_{0,0}(0)$ is observed to decay at the rate of ${\cal O}(M^{-1})$ as $M \to +\infty$.}
\label{fig:blocksize-const}   
\end{center}
\end{figure}

\subsubsection{Piecewise linear weighting functions}
If the coarse-grained variables are defined by the piecewise linear weighting functions, these functions centered in adjacent blocks have overlaps as Fig. \ref{fig: linear} shows, which implies that the operator $\Phi$ is not diagonal and its Fourier transform is no longer a real-valued matrix. It becomes more difficult to simplify \eqref{eq:lakernel}. 

To deal with this case, we first define two vectors ${h}_1$, ${h}_2 \in \mathbb{R}^M$ as follows,
\begin{equation}
{h}_1=\frac{1}{c}\left[ \begin{array}{c} 1 \\ 2 \\ \vdots \\ M \\
\end{array}
\right] \quad \text{ and }\quad
{h}_2=\frac{1}{c}\left[ \begin{array}{c}M\!\!-\!\!1\\ \vdots \\ 1\\ 0
\end{array}
\right],
\end{equation}
where $c=\sqrt{\frac{1}{3}(2M^3+M)}$ is a constant to normalize $\Phi$. $h_{1}$ and $h_{2}$ represent the two sides of the hat in Figure \ref{fig: linear}.

Notice that the operator $\Phi$ is lower block bidiagonal, with the main diagonal block equal to ${h}_1$ and the first off diagonal block is equal to ${h}_2$, i.e.,
\begin{equation}
\Phi= \left[ \begin{array}{ccccc}
\ddots   	  &      	  	& 		&		& 			\\
 \ddots         & 	{h}_1 	&		& 	 	 &		        \\
	           & {h}_2		&{h}_1 	&  		& 			\\
  		  & 			& {h}_2     & {h}_1 	&			 \\
		  & 			& 		&{h}_2	&  \ddots		 \\
		  & 			&		&		& \ddots		 \\
\end{array}
\right].
\end{equation}

We seek $\Psi$ in a lower block bidiagonal form and denote its main diagonal block as $H_1$ and the first off-diagonal block as $H_2$, and both $H_1$ and $H_2$ are in $\mathbb{R}^{M\times(M-1)}$. To make $[\Phi \; \Psi]$ a set of basis of $X=\ell^2(\mathbb{L})$, $H_1$ and $H_2$ should satisfy,
\begin{equation}\left\{
\begin{aligned}
&{h}_1^{\intercal}H_1 + {h}_{2}^{\intercal}H_{2} = 0, \\
&{h}_1^{\intercal}H_2 = 0, \\
&{h}_2^{\intercal}H_1 = 0.
\end{aligned}\right.
\label{eq:choosepsi}
\end{equation}

The above equations hold as long as the column vectors of the matrix $H=\left[
\begin{array}{c}
H_1 \\ H_2 \\
\end{array} \right]$ are the solution set of the following linear equations
\begin{equation}
\left[ \begin{array}{cc}
{h}_1^{\intercal} & {h}_2^{\intercal} \\
{h}_2^\intercal & 0 \\
0 & {h}_1^\intercal \\
\end{array}\right] {y} = 0.
\label{eq:heq}
\end{equation}
whose solution space has dimension $2M-3$.

Meanwhile, we need to ensure that the matrix $H$ has full column rank, i.e., $\text{rank}(H)=M-1$, to guarantee that $\Psi$ spans the whole complementary space of $\Phi$. This can be guaranteed when $M\geq2$.
For example, we let $R\in \mathbb{C}^{M-1}$ be the Cholesky factor of $\Widehat{\Psi}^*\Widehat{\Psi}$, i.e., $R^*R=\Widehat{\Psi}^*\Widehat{\Psi}$. Then we set ${\Phi}_0=\Widehat{\Phi}(\Widehat{\Phi}^*\Widehat{\Phi})^{\text{-}\frac{1}{2}}$ and $\Psi_0=\Widehat{\Psi}R^{-1}$. With this choice, it is straightforward to verify that $[\Phi_0 \; \Psi_0]$ is an orthogonal matrix.

Recall the expression in \eqref{eq:lakernel}, which for the piecewise linear averaging, is reduced to the following expression,
\begin{equation}\begin{aligned}
\Theta_{0,0}(0) &= \frac{1}{2\pi}\int_{0}^{2\pi}\left(\Widehat{\Phi}^{*}\Widehat{\Phi}\right){\Phi}_0^*\Widehat{\cal A}{\Psi}_0{\Big({\Psi}_0^{*}\Widehat{\cal A}{\Psi}_0\Big)}^{-1}{\Psi}_0^{*}\Widehat{\cal A}{\Phi}_0\text{~d}\xi.
\end{aligned}\label{eq:linearla}
\end{equation}

\begin{theorem}If the coarse-grained variables are defined with piecewise linear weighting functions, then the following estimate holds,
\begin{equation}
\Theta_{0,0}(0) \leq \frac{2M\kappa_{1}+8M\kappa_{2}-6\kappa_{2}}{M(2M^{2}+1)/3}.
\end{equation}
\label{thm: 2}
\end{theorem}

\begin{proof}
Applying Lemma \ref{lm1} to ${\Big({\Psi}_0^{*}\Widehat{\cal A}{\Psi}_0\Big)}^{{-}1}$ and following similar steps in the proof of Theorem \ref{thm:1}, we have
\begin{equation}\begin{aligned}
\Theta_{0,0}(0) &= \frac{1}{2\pi}\int_{0}^{2\pi} \Widehat{\Phi}^*\Widehat{\cal A}{\Widehat{\Phi}}-\Big(\Widehat{\Phi}^{*}\Widehat{\cal A}^{-1}{\Widehat{\Phi}}\Big)^{-1}\left( \Widehat{\Phi}^*\Widehat{\Phi}\right)^2
 \text{~d}\xi.
\end{aligned}\label{eq:eb}\end{equation}

Again, with the positive semidefinite property, it suffices to estimate the integral of the first term. Direct calculations yield, 
\begin{equation}\begin{aligned}
\frac{1}{2\pi}\int_{0}^{2\pi} \Widehat{\Phi}^{*}\Widehat{\cal A}\Widehat{\Phi}\text{~d}\xi 
& = \frac{1}{2\pi}\int_{0}^{2\pi}
\left( {h}_{1}^{\intercal}+{h}_{2}^{\intercal}e^{-i\xi}\right)\left( A_{0}+A_{1}e^{i\xi}+A_{1}^{\intercal}e^{-i\xi}\right)\left( {h}_{1}+{h}_{2}e^{i\xi}\right)
 \text{~d}\xi \\
&=\frac{1}{2\pi}\int_{0}^{2\pi}{h}_{1}^{\intercal}A_{0}{h}_{1}+2{h}_{1}^{\intercal}A_{1}^{\intercal}{h}_{2}+{h}_{2}^{\intercal}A_{0}{h}_{2}\text{~d}\xi \\
&=\frac{2M\kappa_{1}+8M\kappa_{2}-6\kappa_{2}}{M(2M^{2}+1)/3}.
\end{aligned}\label{eq:la}\end{equation}
Then we have the estimate,
\begin{equation}
 \Theta_{0,0}(0) \leq \frac{1}{2\pi} \int_{0}^{2\pi} \Widehat{\Phi}^*\Widehat{\cal A}{\Widehat{\Phi}} \text{~d}\xi = \frac{2M\kappa_{1}+8M\kappa_{2}-6\kappa_{2}}{M(2M^{2}+1)/3}.
\end{equation}
\end{proof}
\qed

\noindent{\bf Remark:} Interestingly, with this coarse-graining scheme, $ \Theta_{0,0}(0) $, the variance of the random noise, decays more rapidly with rate $\mathcal{O}(M^{-2})$ compared to constant averaging. 

With the same setup as in the previous numerical test, we tested the dependence of $\Theta_{0,0}(0)$ on $M$ for piecewise linear averaging numerically. The results, as shown in Figure \ref{fig:blocksize-linear}, again agree with the estimate.
\begin{figure}[htp]
\begin{center}
\includegraphics[width=0.45\textwidth,height=0.35\textwidth]{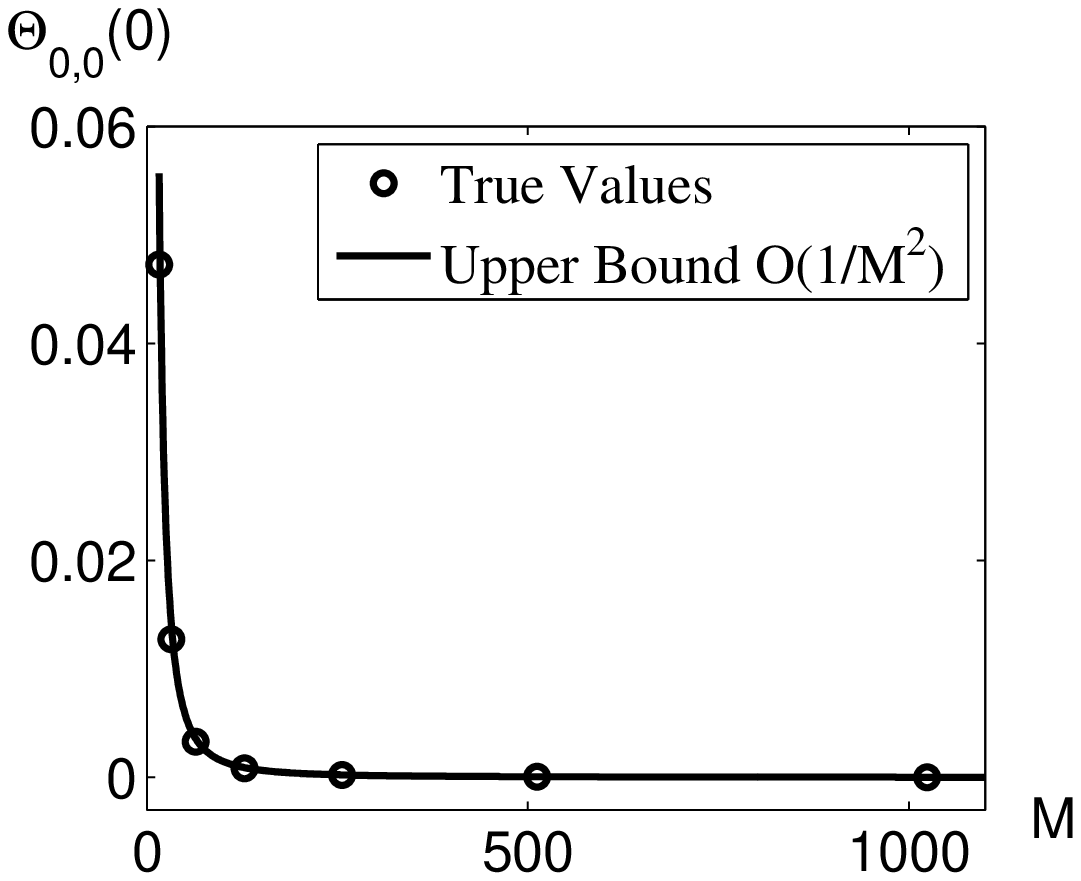}
\includegraphics[width=0.45\textwidth,height=0.35\textwidth]{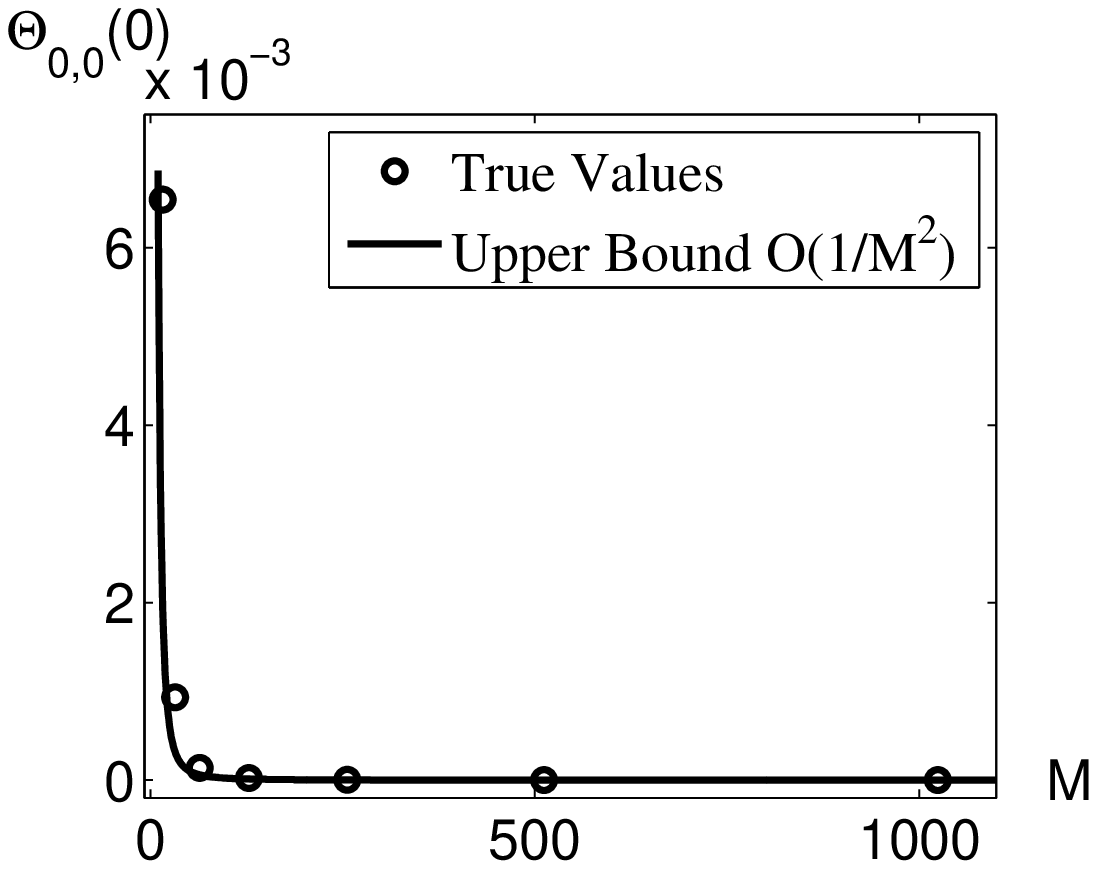}
\caption{The figures show the dependence of the kernel function on the block size for piecewise linear averaging operator with different force constants. We fix the total number of the atoms $N=2^{13}$ in the simulation and vary the block size $M=2^{4},2^{5},\cdots,2^{10}$, which represents the averaging scale. For the left figure, the force constants are obtained from a Morse potential, with $\kappa_{1} = 12.2676$ and $\kappa_{2} = 3.0628$. For the right figure, $\kappa_{1} = 12.2676$ and $\kappa_{2} = -3$, which is close to the boundary of stability conditions \eqref{eq:stability}. $\Theta_{0,0}(0)$ in both cases show decay at rate of $\cal{O}(M^{-2})$ as anticipate.}
\label{fig:blocksize-linear}   
\end{center}
\end{figure}

\subsection{The spatial decay of $\Theta_{0,J}(0)$ }
Now we turn to the {\it off-diagonal} entries of the $\Theta(0)$. Only the piecewise constant averaging case is considered here. 

Recall from \eqref{eq:lakernel} that we have
\begin{equation} 
\Theta_{0,J}(0) = \frac{1}{2\pi}\int_{0}^{2\pi} e^{iJ\xi}Q_{1}^{\intercal}\Widehat{\cal A}Q_{2}\left( Q_{2}^{\intercal}\Widehat{\cal A}Q_{2}\right)^{-1} Q_{2}^{\intercal}\Widehat{\cal A}Q_{1} \text{~d}\xi.
\label{eq:Theta0J}
\end{equation}

In light of the second fluctuation-dissipation theorem \eqref{eq: 2FDT}, this entry indicates the spatial correlation of the random noise. For this analysis, we will fix the block size $M$, and focus on the behavior of the kernel function as $|J| \gg 1.$

\begin{lemma}
Let $F(\xi) = \left(Q_{1}^{\intercal}\Widehat{\cal A}^{-1}Q_{1}\right)^{-1}$ on $(0,2\pi]$, then $F(\xi)$ is a real-valued even function and it is $n$ times continuously differentiable for any integer $n$, i.e., $F \in C^{\infty}\left( 0,2\pi \right]$.
\end{lemma}
\begin{proof}
Since $\Widehat{\cal A}$ is Hermitian, it is clear that $F$ is real-valued and even. For its differentiability, let's first consider $\xi \neq 0$ when $\Widehat{\cal A}$ is invertible. From Cramer's rule, one has
\begin{equation}\begin{aligned}
F(\xi) = \text{det}(\Widehat{\cal A}) \left( Q_{1}^{\intercal} \text{adj}(\Widehat{\cal A})Q_{1}\right)^{-1},
\end{aligned}\end{equation}
where $\text{det}(\Widehat{\cal A})$ is the determinant of $\Widehat{\cal A}$ and $\text{adj}(\Widehat{\cal A})$ is the adjugate matrix of $\Widehat{\cal A}$. This expression also shows that $\xi=0$ is not a singularity to $F(\xi)$. From Lemma \ref{lemma:eigenA}, we know that $\text{det}(\Widehat{\cal A})$ is smooth w.r.t $\xi$, and next we will prove $\left( Q_{1}^{\intercal} \text{adj}(\Widehat{\cal A})Q_{1}\right)^{-1}$ is smooth as well. 

From direct computations, we observe that
\begin{equation}
\left( Q_{1}^{\intercal} \text{adj}(\Widehat{\cal A})Q_{1}\right)^{-1} = \Big( C_{0}-C_{1}\cos\xi-C_{2}\cos 2\xi \Big)^{-1},
\label{eq: adjA}
\end{equation}
where $C_{0},C_{1}$ and $C_{2}$ are constants related to $\kappa_{1}$ and $\kappa_{2}$, and $C_{0}>|C_{1}|+|C_{2}|$. This implies that $\left( Q_{1}^{\intercal} \text{adj}(\Widehat{\cal A})Q_{1}\right)$ is positive and both $\left( Q_{1}^{\intercal} \text{adj}(\Widehat{\cal A})Q_{1}\right)^{-1}$ and $F(\xi)$ belong to $C^{\infty}\left( (0,2\pi] \right)$.
\end{proof}
\qed

\begin{theorem} When the coarse-grained variables are defined by piecewise constant weighting functions, then
\begin{equation}
|\Theta_{0,J}(0)| \leq {o}\left( |J|^{-n} \right)
\end{equation}
as $|J|\rightarrow \infty$ for any positive integer $n$.
\end{theorem}
\begin{proof}
Similar to \eqref{eq:est1}, we can simplify the expression \eqref{eq:Theta0J} as follows,
\begin{equation}
\Theta_{0,J}(0) = \frac{1}{2\pi}\int_{0}^{2\pi}e^{iJ\xi}\left[ Q_{1}^{\intercal}\Widehat{\cal A}Q_{1} - (Q_{1}^{\intercal}\Widehat{\cal A}^{-1}Q_{1})^{-1}\right] \text{ d}\xi.
\label{eq: theta0J0}
\end{equation}
Direct calculation yields,
\begin{equation}
\frac{1}{2\pi}\int_{0}^{2\pi}e^{iJ\xi}Q_{1}^{\intercal}\Widehat{\cal A}Q_{1}\text{ d}\xi = \frac{\kappa_{1}+2\kappa_{2}}{\pi}\int_{0}^{2\pi}\cos(J\xi)\left( 1-\cos(\xi)\right) \text{ d}\xi.
\end{equation}
When $J$ is an integer greater than $1$, the integral is zero. 

Notice that $\left(Q_{1}^{\intercal}\Widehat{\cal A}^{-1}Q_{1}\right)^{-1}$ is even and $F(\xi)$ is periodic. By integration by parts repeatedly, one can have the estimate,
\begin{equation}\begin{aligned}
\Theta_{0,J}(0) &= -\frac{1}{2\pi}\int_{0}^{2\pi} F(\xi)\cos(J\xi) \text{ d}\xi \\
&=\frac{1}{J} \left[ \frac{1}{2\pi}\int_{0}^{2\pi} F'(\xi)\sin(J\xi) \text{ d}\xi\right] \\
&= \frac{1}{J^{2}} \left[ \frac{1}{2\pi}\int_{0}^{2\pi} F''(\xi)\cos(J\xi) \text{ d}\xi \right]\\
&= \frac{1}{J^{3}} \left[ -\frac{1}{2\pi}\int_{0}^{2\pi} F'''(\xi)\sin(J\xi) \text{ d}\xi \right]\\
&= \quad \cdots
\end{aligned}
\label{eq:spaceest}
\end{equation}
So $|\Theta_{0,J}(0)|\leq o(|J|^{-n})$ for any integer order $n$.
%
%
\end{proof}
\qed

To verify the fast decay, we used the same model as the previous numerical tests, and plotted the kernel function v.s. $J$ on a log scale. This is shown in Figure \ref{fig:spacedecay}. In fact, we observe an exponential decay.
\begin{figure}[h]
\begin{center}
\includegraphics[width=0.55\textwidth,height=0.45\textwidth]{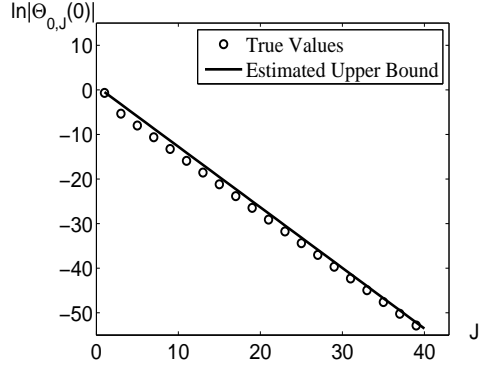}
\caption{This figure shows the spatial decay of the entry $\Theta_{0,J}$ on the space. With the fixed total number of atoms $N=2^{15}$ and averaging block size $M=2^{5}$, we observe the dependence of $\Theta_{0,J}(0)$ on space index $J$ for the piecewise constant coarse-grained averaging operator and it shows exponential decay. The $x$-axis indicates the matrix index $J$ and the $y$-axis is the logarithms of the absolute values of $\Theta_{0,J}(0)$. 
}
\label{fig:spacedecay}   
 \end{center}   
\end{figure}

\subsection{The dependence of $\Theta_{0,0}(t)$ on time $t$}
Finally, we analyze how the kernel function changes in time. Here we will focus on the time decay estimate when the averaging operators are piecewise constant. Namely our aim is to estimate the asymptotic behavior of
\begin{equation}
\Theta_{(0,0)}(t) = \frac{1}{2\pi}\int_{0}^{2\pi} \Widehat{\Phi}^*{\Widehat{\cal A}}\Widehat{\Psi}\cos(\Widehat{\Omega} t)\Widehat{\Omega}^{-2}\Widehat{\Psi}^{*}\Widehat{\cal A}\Widehat{\Phi} ~d\xi,
\label{eq:estt}
\end{equation}
as $t \to +\infty$.

\begin{theorem} When the coarse-grained variables are defined by piecewise constant weighting functions, the kernel function decays in time with rate at least equal to $0.5$. Namely,
 \begin{equation}
\left|\Theta_{0,0}(t)\right| = \mathcal{O}\left(t^{-1/2}\right) \; {\rm as } \;\;t \rightarrow +\infty.
\end{equation}
\label{thm: timedecay}
\end{theorem}

Let $\Widehat{\Psi}^{*}\Widehat{\cal A}\Widehat{\Psi}=X^{*}M^{2}X$ be the eigen decomposition of $\Widehat{\Psi}^{*}\Widehat{\cal A}\Widehat{\Psi}$, where $M$ is a real diagonal matrix and $X$ is a unitary matrix. Then we are able to write \eqref{eq:estt} as
\begin{equation}
\Theta_{0,0}(t) = \frac{1}{2\pi}\int_{0}^{2\pi} \sum_{j=0}^{M-1} \frac{\cos(\mu_{j}t)|g_{j}|^{2}}{\mu_{j}^{2}}~d\xi,
\label{eq:fourierint}
\end{equation}
where $\mu_{j}=M_{j,j}$ is an eigenvalue of $\Widehat{\Omega}$ and $g_{j}$ is the $j$th component of $X\Widehat{\Psi}^{*}\Widehat{\cal A}\Widehat{\Phi}$. 

Notice that
\begin{equation}
\Theta_{0,0}(t) = \sum_{j=0}^{M-1} \text{Re} \left\{ \frac{1}{2\pi} \int_{0}^{2\pi} \frac{|g_{j}|^{2}}{\mu_{j}^{2}}e^{i\mu_{j}t} ~d\xi \right\}.
\end{equation}
So this is an integral of Fourier type and we are regarding $t$ as a large parameter. In this regime, the integrand oscillates severely and causes cancellation. The stationary phase approximation method can be used here to estimate the decay rate of \eqref{eq:fourierint}.

The stationary phase method is an approach for estimating integrals with fast oscillation by evaluating the integrands in regions where they contribute the most. Let's recall the stationary phase approximation method \cite{bleistein1975asymptotic} with following lemma.

\begin{lemma}
Consider the integral,
\begin{equation}
I(t) = \int_{a}^{b} f(x)e^{i\omega(x)t} \text{ d}x ,
\end{equation}
where $f$ and $\omega$ are real-valued functions. $\omega$ is called phase function. Assume $\omega\in C^{2}(a,b)$ and has one stationary point at $x=x_{0}\in (a,b)$, where $\omega'(x_{0})=0$ and $\omega''(x_{0})\neq 0$. Then the integral can be approximated asymptotically by
\begin{equation}
I(t) \approx e^{i\omega(x_{0})t}e^{\text{sign}(\omega''(x_{0}))\frac{i\pi}{4}} f(x_{0})\left[ {\frac{2\pi}{t|\omega''(x_{0})|}} \right]^{\frac{1}{2}}
\end{equation}
as $t \rightarrow \infty$.
\end{lemma}

Up to now, we only know $\left\{ \mu_{j}(\xi) \right\}_{j=1}^{M-1}$ as the eigenvalues of $\Widehat{\Psi}^{*}\Widehat{\cal A}\Widehat{\Psi}$ for given $\xi$. Applying the stationary phase method would require the presence of stationary points. For this purpose, we need to label the eigenvalues with proper indices, so that the function $\mu_{j}(\xi)$ is continuous (or even smooth depending on $\Widehat{\cal A}$) w.r.t $\xi$. $\mu_{j}$ will be called branches of eigenvalues, as depicted in Figure \ref{fig: eig3}. 
\begin{figure}[h]
\begin{center}
\includegraphics[width=0.45\textwidth]{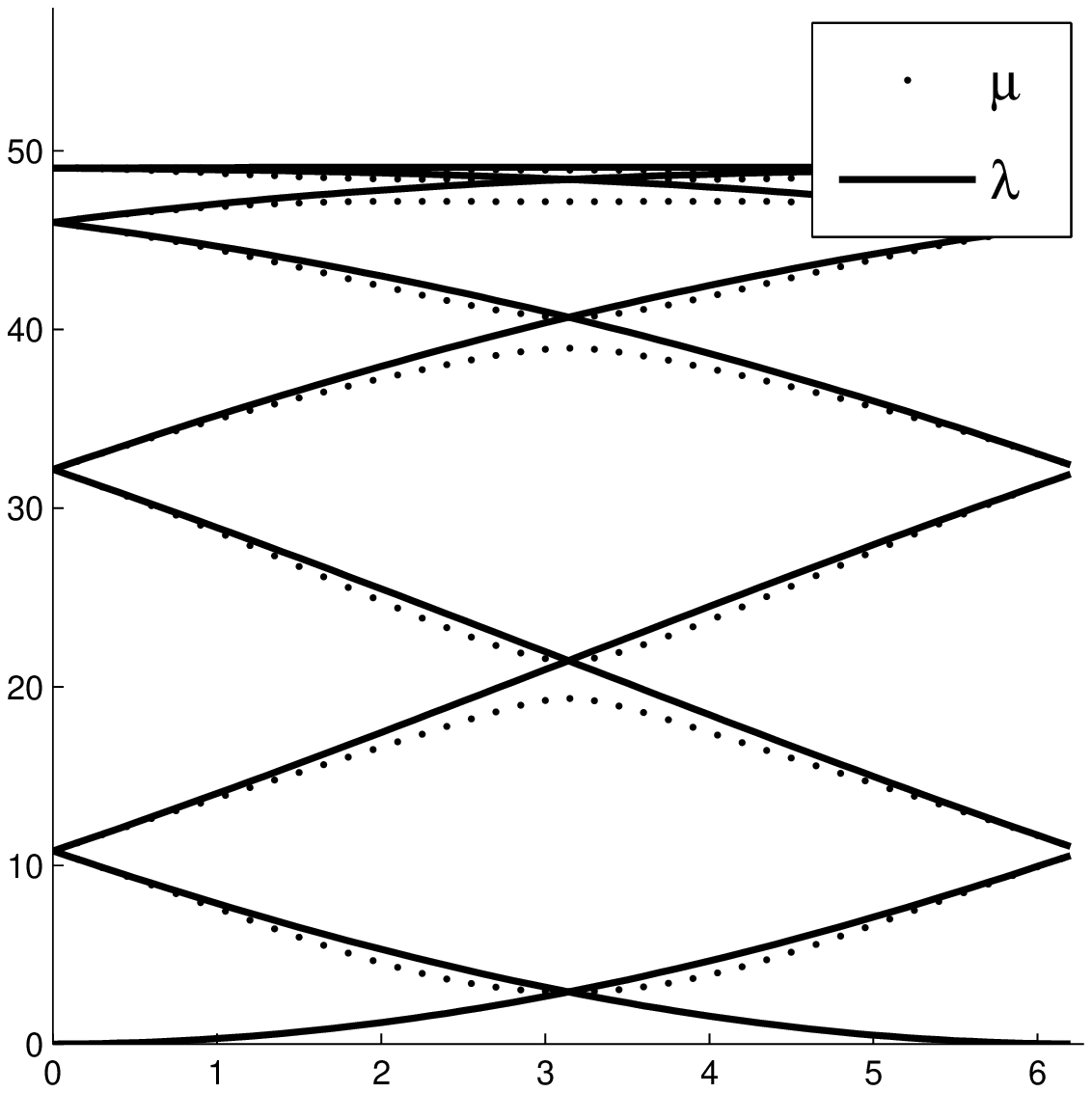}
\includegraphics[width=0.45\textwidth]{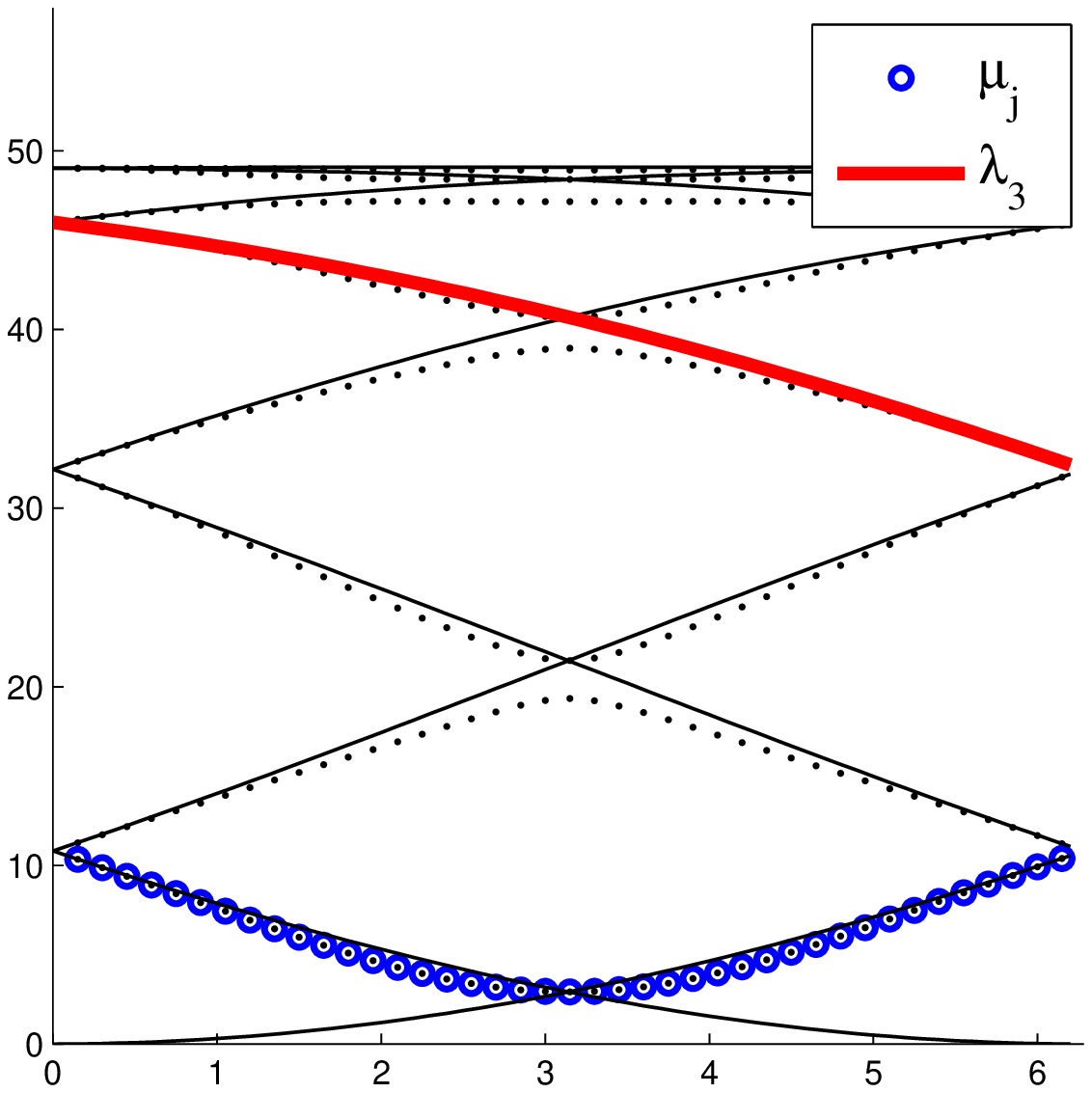}
\caption{These figures show eigenvalues of $\Widehat{\Psi}^{*}\Widehat{\cal A}\Widehat{\Psi}$ labelled with $\mu$ and eigenvalues of $\Widehat{\cal A}$ labelled with $\lambda$ when $M=9$. It shows the interlacing properties of the eigenvalues. The right figure highlights two branches of the eigenvalues of $\Widehat{\Psi}^*\Widehat{\cal A}\Widehat{\Psi}$ and $\Widehat{\cal A}$ respectively. 
}
\label{fig: eig3}
\end{center}
\end{figure}

Now let's take a further look at eigenvalues of $\Widehat{\Psi}^{*}\Widehat{\cal A}\Widehat{\Psi}$. Let $\mu$ be an eigenvalue of $\Widehat{\Psi}^* \Widehat{\cal A} \Widehat{\Psi}$ and ${x}$ be its corresponding eigenvector, i.e.,
\begin{equation}
 \Widehat{\Psi}^* \Widehat{\cal A} \Widehat{\Psi} {x} = \mu {x}, \quad |{x}| = 1.
 \label{eq:eigenprob}
\end{equation}

In the proof of Theorem \ref{thm: 2}, we know $\Widehat{\Psi}^{*}\Widehat{\cal A}\Widehat{\Psi}$ is non-sigular, so $\mu \neq 0$. In addition, for any two matrices $C\in \mathbb{C}^{k\times l}$ and $D \in \mathbb{C}^{l \times k}$, $CD$ and $DC$ have the same eigenvalues (counting multiplicities) except for the zero eigenvalue \cite{horn2012matrix}. In other words, $\mu$ is also a non-zero eigenvalues of $\Widehat{\cal A} \Widehat{\Psi} \Widehat{\Psi}^*$. So the original eigenvalue problem \eqref{eq:eigenprob} is equivalent to 
\begin{equation}
\Widehat{\cal A} \Widehat{\Psi} \Widehat{\Psi}^* {y} = \mu {y}, \quad \mu \neq 0,
\quad |{y}| = 1. 
\label{eq2}
\end{equation}

Let $\Widehat{\cal A}=V\Lambda V^{*}$ be the eigen decomposition of $\Widehat{\cal A}$, then the eigenvalue problem is equivalent to solving $\mu$ in the equations as follows,
\begin{equation}
 \left( \Lambda - \Lambda V^* \Widehat{\Phi} \Widehat{\Phi}^* V \right) \widetilde{{y}} = \mu \widetilde{{y}}, \quad \mu \neq 0,
 \label{eq:eigenB}
\end{equation}
where $\widetilde{{y}}=V^{*}{y}$.

The matrix determinant Lemma {\cite{matrixlemma}} states that if $B$ is a complex square matrix and ${u},{v}$ are column vectors, then the following identity holds,
\begin{equation}
 \det (B + {u}{v}^*) = \det (B) + {v}^* \text{adj}(B) {u},
\end{equation}
where $\det(B)$ is the determinant of $B$ and $\text{adj}(B)$ is the adjugate matrix of $B$. To deal with the eigenvalue problem in \eqref{eq:eigenB}, let the matrix $B=\Lambda - \Lambda V^* \Widehat{\Phi} \Widehat{\Phi}^* V$ and apply the matrix determinant Lemma to it. We would have the characteristic polynomial,
\begin{equation}
f(\mu) = \text{det}(\mu I - B) =\det (\mu I - \Lambda) - {z}^* \text{adj}(\mu I - \Lambda) \Lambda {z},
\label{eq:det}
\end{equation}
where ${z} =V^*\Widehat{\Phi}= \frac{1-e^{i \xi}}{M} \left( \frac{1}{1-e^{i\xi_0}}, \frac{1}{1-e^{i\xi_1}}, \cdots, \frac{1}{1-e^{i\xi_{M\text{-}1}}} \right)^*$. Setting the above characteristic polynomial equal to zero, we have
\begin{equation}
 \prod_{j = 0}^{M-1}
 (\mu - \lambda_j) - \sum_{j = 0}^{M-1} \lambda_j | z_j |^2 \left( \prod_{\substack{k=0 \\ k\neq j}}^{M-1}
 (\mu - \lambda_k)\right) = 0, \label{eq:10}
\end{equation}
where $\lambda_j = 2 \kappa_{1} (1 - \cos \xi'_j) + 2 \kappa_{2} (1 - \cos 2 \xi'_j)$ are the eigenvalues of $\Widehat{\cal A}$. 

From the fundamental Theorem of Algebra one can define $\mu(\xi)$ as a branch of non-trivial solutions to \eqref{eq:10}. Together with the explicit form of $\lambda_{j}$ and $z_{j}$, we know $\mu(\xi)$ is smooth.

We may enforce an increasing order to label eigenvalues, for instance, $\mu_{0}\leq \mu_{1} \leq \cdots \leq \mu_{M-2}$. However, when two branches intersect with each other, this labelling may break down the smoothness of branches. Figure \ref{fig: eig} further demonstrates one possible scenario.

Fortunately, we have, 
\begin{figure}[h]
\begin{center}
\includegraphics[width=0.4\textwidth,height=0.3\textwidth]{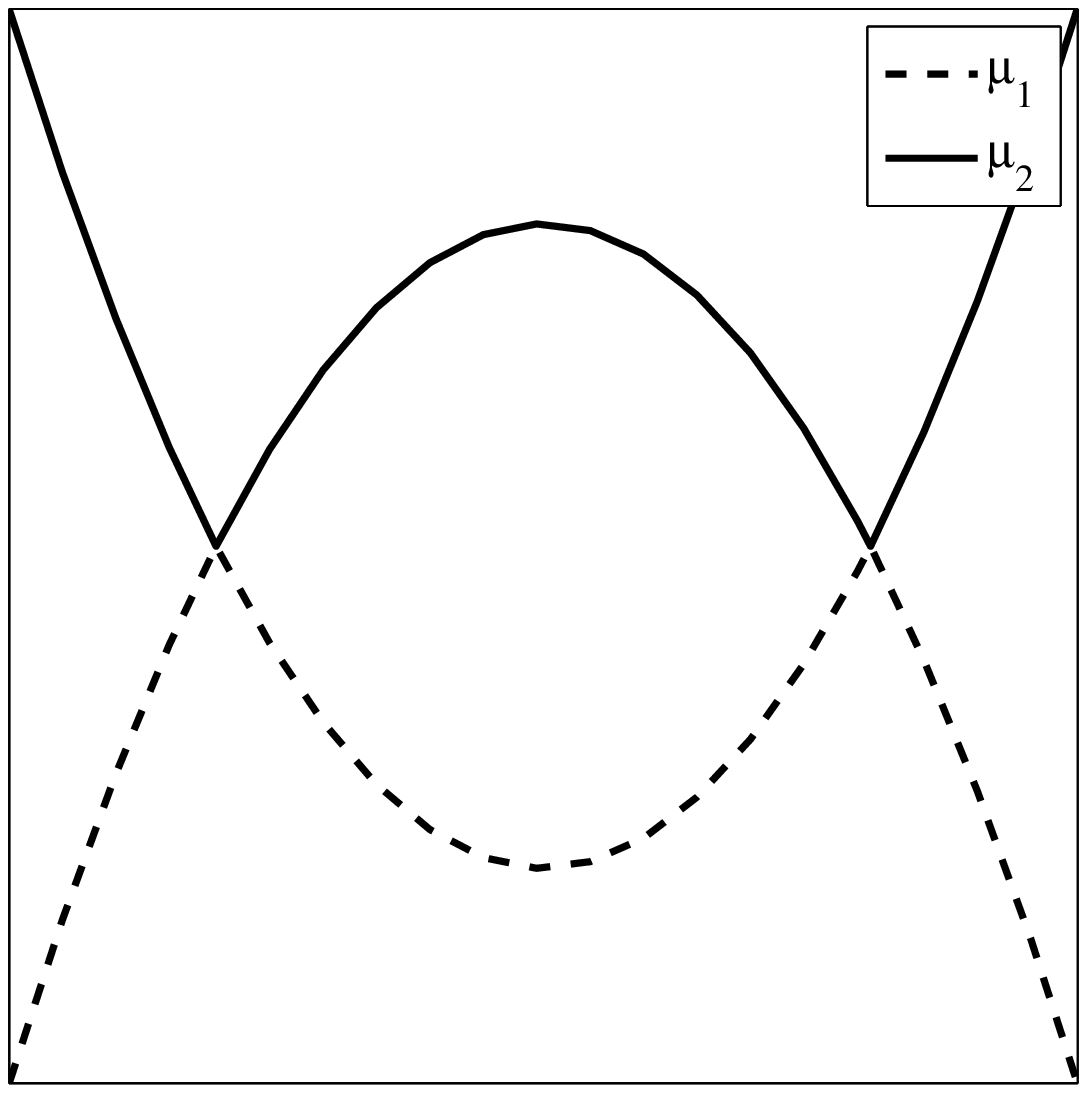}
\includegraphics[width=0.4\textwidth,height=0.3\textwidth]{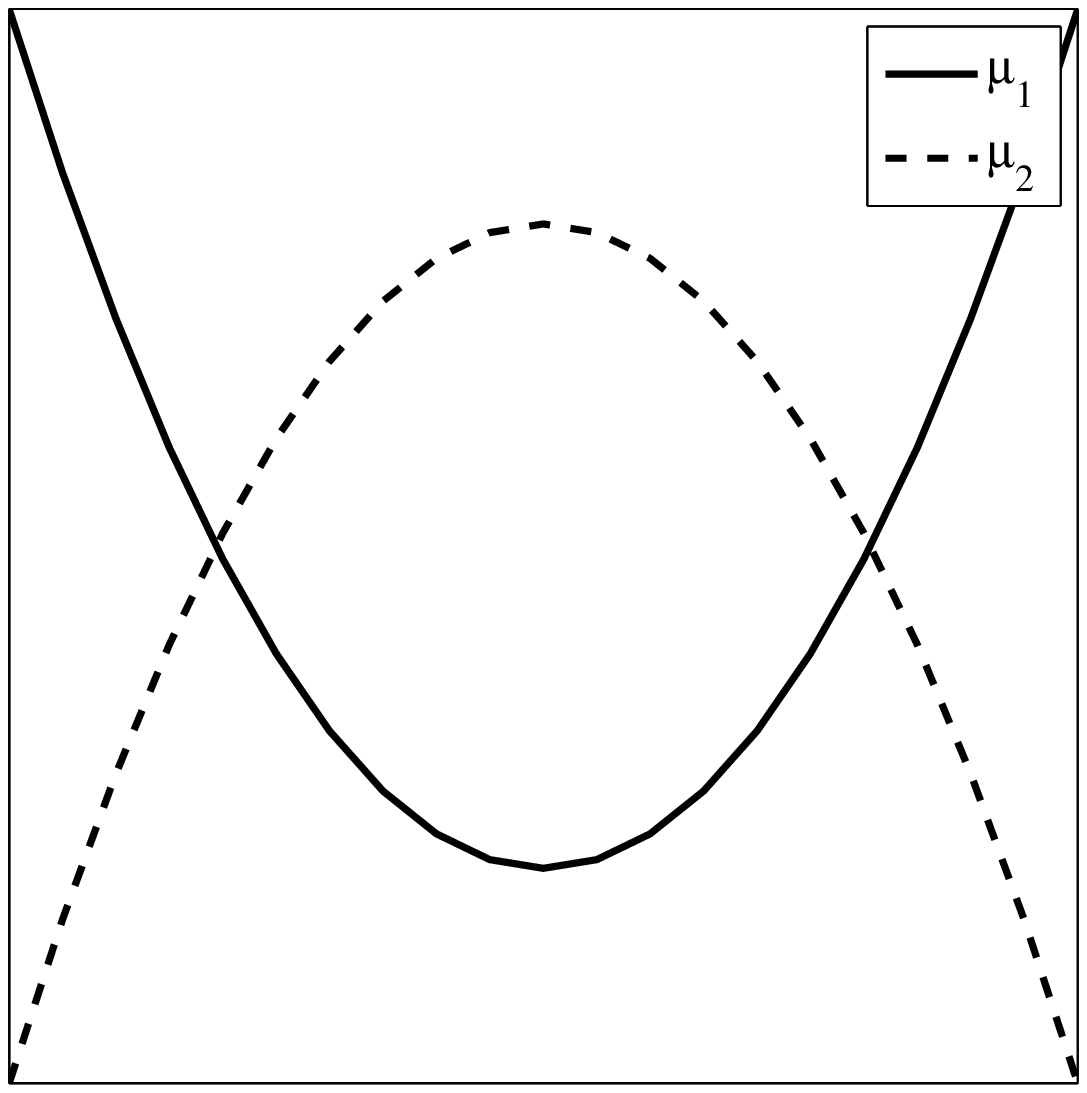}
\caption{Both figures have the same two branches of eigenvalues but labelled in different ways. Eigenvalues in the left figure are labelled using the way described in the context satisfying $\mu_{1}\leq \mu_{2}$ but it doesn't guarantee global smoothness. The right figure shows a more natural way to number eigenvalues and it preserves the global smoothness.}
\label{fig: eig}
\end{center}
\end{figure}

\begin{lemma}
For $M>4$ and $\xi\in (0,2\pi)$, the matrix $\Widehat{\Psi}^{*}\Widehat{\cal A}\Widehat{\Psi}$ does not have repeated eigenvalues.
\label{lemma3}
\end{lemma}

\begin{proof}
Suppose on the contrary at $\xi=\eta$, the matrix $\Widehat{\Psi}^{*}\Widehat{\cal A}\Widehat{\Psi}$ has a repeated eigenvalue $c$. This means that for some $j,$ $\mu_{j-1}=\mu_{j}=c$. Recall $\lambda_{j}$'s are eigenvalues of $\Widehat{\cal A}$ and explicitly given in Lemma \ref{lemma:eigenA}. From Cauchy interlacing theorem {\cite{interlacing}} for eigenvalues, at $\xi=\eta$, one has 
\begin{equation}
\lambda_{i_{0}} \leq \mu_{0} \leq \lambda_{i_1} \leq \mu_{1} \leq \cdots \leq \lambda_{i_{M-2}}\leq \mu_{M-2} \leq \lambda_{i_{M-1}},
\label{eq: interlacing}
\end{equation}
where $\{{i_j}\}_{j=0}^{M-1}$ is a permutation of $\{{j}\}_{j=0}^{M-1}$ subject to $\lambda_{i_{j-1}}\leq \lambda_{i_j}$. Then it is clear that $\lambda_{i_j}=c$ as well. 

With a substitution of $\mu=c$ and $\xi=\eta$ into \eqref{eq:10}, one has,
\begin{equation}
	w_{i_{j}}(\eta)\prod_{\substack{k=0 \\ k\neq i_{j}}}^{M-1} (c-\lambda_{k}(\eta)) = 0,
\end{equation}
where $w_{j}(\xi)=\lambda_j |z_j|^2 = \frac{2 (1 - \cos \xi)}{M^2}[\kappa_{1} + 2 \kappa_{2} (1 + \cos\xi'_j)]$. 

Since $\eta\neq 0$, $w_{j}(\eta)$ is a positive number, so there must be at least one term in the product equal to zero and let $\lambda_{l}(\eta) = c~(l\neq i_{j})$, then one can factor $\mu-c$ out from \eqref{eq:10} and evaluate the equation at $\mu=c$ and $\xi=\eta$, 
\begin{equation}
(w_{i_{j}}(\eta)+w_{l}(\eta))\prod_{\substack{ k=0\\ k\neq i_{j},l}}^{M-1}(c-\lambda_{k}) = 0.
\end{equation}
Due to the fact that $w_{i_{j}}(\eta)$ and $w_{l}(\eta)$ are positive, there exists another index $h$ different from $i_{j}$ and $l$ s.t. $\lambda_{h}(\eta)=c$. Repeating above steps, one will have all eigenvalues are equal to $c$, which is impossible as explained next.


Let's write $\lambda_{k}=-4\kappa_{2}\cos^{2}\xi'_{k}-2\kappa_{1}\cos\xi'_{k}+2\kappa_{1}+4\kappa_{2},~\xi'_{k} = \frac{-\xi+2j\pi}{M}
$ as a quadratic function,
\begin{equation}
f(x)=-4\kappa_{2}x^{2}-2\kappa_{1}x+2\kappa_{1}+4\kappa_{2},
\end{equation} 
when $x=\cos\xi'_{k}$. 

When the block size $M$ is greater than $4$, there must be at least 3 of $\{ x_{0}, x_{1}, \cdots, x_{M-1}\}$ equal to a same number, $d$. However, the equation
\begin{equation}
\cos\xi'_{j}=\cos\frac{-\xi+2j\pi}{M}=d,
\end{equation}
has $2$ different solutions at most when $\xi\in(0,2\pi)$. This implies the contradiction. 
\qed
\end{proof}

With the help of Lemma \ref{lemma3}, it's easy to prove Theorem \ref{thm: timedecay}.

\begin{proof}
From Lemma \ref{lemma3} and \eqref{eq:10}, one can see branches of eigenvalues are well defined and don't intersect with each other. Since we have $\mu_{j}(\xi)=\mu_{j}(2\pi-xi)$ for each branch, $\mu'_{j}(\pi)=0$ is a stationary point on the interval $(0,2\pi)$. According to the stationary phase method, $\Theta_{0,0}(0)$ is bounded at least by ${\cal{O}}\left( t^{-1/2}\right)$ as $t\rightarrow \infty.$

\end{proof}
\qed
\

\noindent{\bf Remark:} Based on numerical observations, $\mu_{1}''(\pi)$ is always non-zero, so this estimate can not be improved.

Numerical simulations about the dependence of the kernel function on time are shown in Figure \ref{fig:timedecay}.
\begin{figure}[htp]
\begin{center}
\includegraphics[width=0.45\textwidth,height=0.35\textwidth]{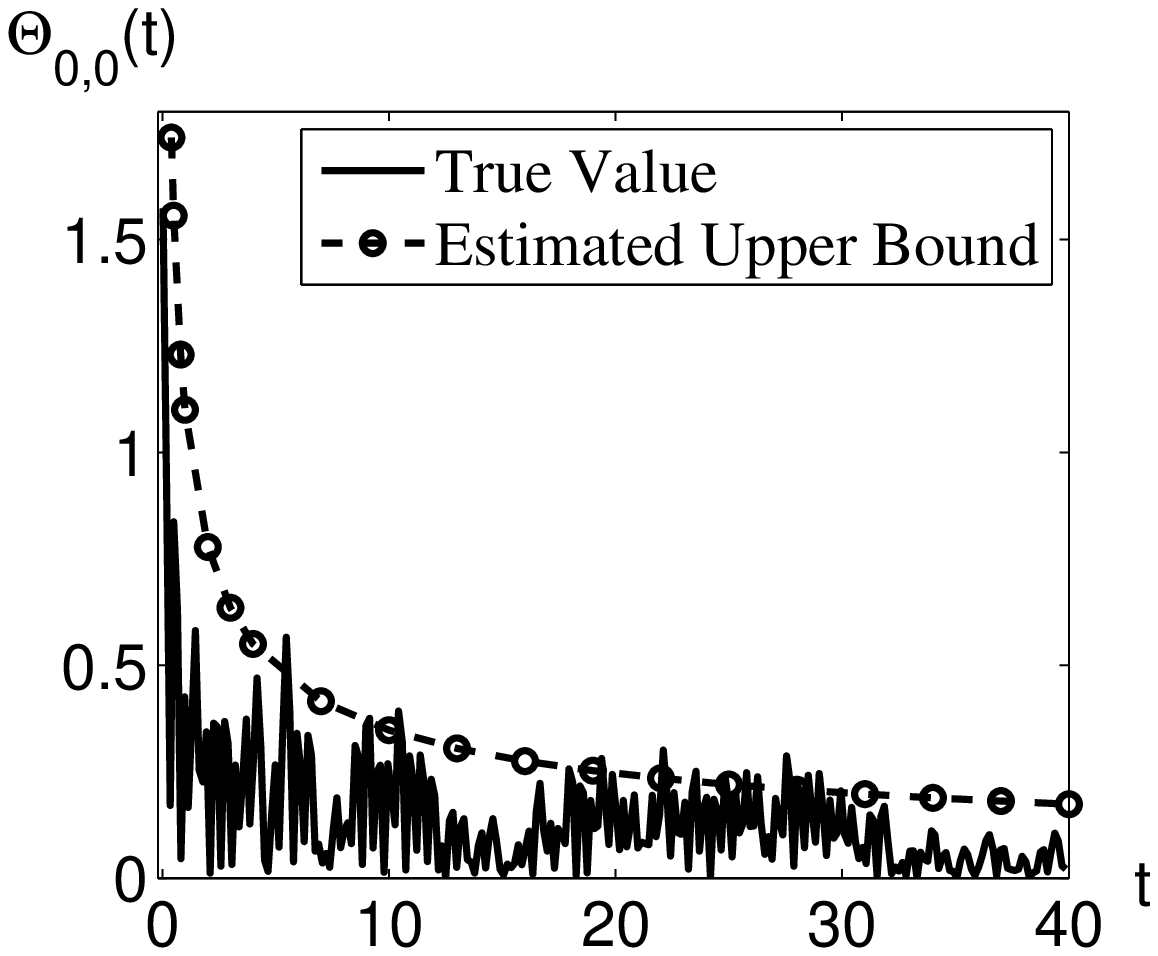}
\includegraphics[width=0.45\textwidth,height=0.35\textwidth]{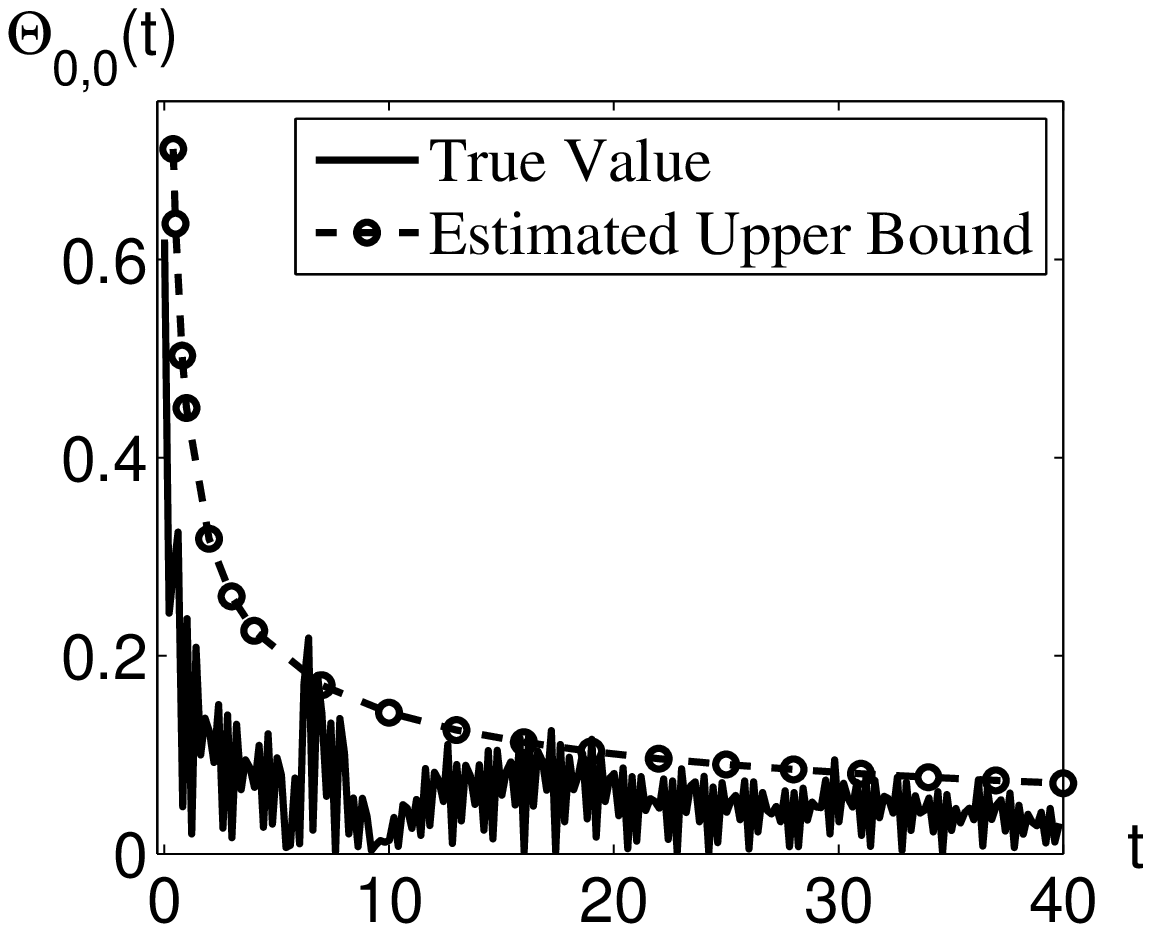}
\caption{These figures show the diagonal entry of kernel function $\Theta_{0,0}(t)$ decays as time approaches infinity with different force constants. $2000$ atoms are taken into consideration and are divided into $20$ blocks evenly. The left figure shows the results when $\kappa_{1} = 12.2676$ and $\kappa_{2} = 3.0628$ and the right one shows the results when $\kappa_{1} = 12.2676$ and $\kappa_{2} = -3$. $\Theta_{0,0}(t)$ is compared with functions in the order of $t^{-1/2}$.}
\label{fig:timedecay}  
 \end{center}     
\end{figure}

\section{Summary and Discussions}
The generalized Langevin equations have recent emerged as a new modeling paradigm as an efficient alternative to full-atom simulations. In principle, the kernel function in the generalized Langevin equations has to come from the underlying atomistic model. In the case of harmonic interactions, the kernel function can be expressed explicitly in terms of a matrix-valued function. This paper analyzed several decay properties. In particular, we found that the magnitude of the kernel function depends on the level of coarse-graining. This suggests that both the memory and random force terms are scale-dependent. 

On the other hand, our analysis also shows that the correlation between the coarse-grained variables decays rapidly with respect to the distance. This supports the typical practice where a cut-off radius is introduced in the modeling of the interactions of the coarse-grained variables. In sharp contrast to the spatial decay, the time decay is rather slow. In fact, the generalized Langevin equation derived by Adelman and Doll \cite{AdDo76,AdDo74}, as well as the model used to fit the experimental data \cite{min2005observation}, also have kernel functions with the similar decay rate. 

The slow decay in the time poses a rather challenging problem for the approximations of the kernel functions. For example, the early t-model approximation \cite{ChSt05} will only be valid for a short time period \cite{zhu2017rigorous}. The Markovian approximation \cite{hijon2010mori,hijon2006markovian} would capture the long-time statistics with a Green-Kubo type of approximations, but the accuracy at the transient time scale is not guaranteed. Beylkin and Monz{\'o}n \cite{beylkin2010approximation} proved a striking result that shows that a power law function can be approximated uniformly by a sum of exponentials for all $t>0.$ This type of approximations have been used in \cite{baczewski2013numerical,fricks2009time,lei2016data,Ma2005}. But no specific error estimate has been provided so far.

A natural extension of the current work is to high dimensional systems. In general, the matrix form of the kernel function still holds. But it is not clear whether the decay rate would be the same. 
 
\bibliographystyle{plain}
\bibliography{gle}

\begin{thebibliography}{10}

\bibitem{AdDo74}
S.~A. Adelman and J.~D. Doll.
\newblock Generalized {Langevin} equation approach for atom/solid-surface
  scattering: Collinear atom/harmonic chain model.
\newblock {\em J. Chem. Phys.}, 61:4242, 1974.

\bibitem{AdDo76}
S.~A. Adelman and J.~D. Doll.
\newblock Generalized {Langevin} equation approach for atom/solid-surface
  scattering: General formulation for classical scattering off harmonic solids.
\newblock {\em J. Chem. Phys.}, 64:2375, 1976.

\bibitem{ashcroft1976solid}
N.W. Ashcroft and N.D. Mermin.
\newblock {\em Solid State Physics}.
\newblock HRW international editions. Holt, Rinehart and Winston, 1976.

\bibitem{baczewski2013numerical}
Andrew~D Baczewski and Stephen~D Bond.
\newblock {Numerical integration of the extended variable generalized
  {Langevin} equation with a positive {Prony} representable memory kernel}.
\newblock {\em The Journal of chemical physics}, 139(4):044107, 2013.

\bibitem{berkowitz1981memory}
Max Berkowitz, John~D Morgan, Donald~J Kouri, and J~Andrew McCammon.
\newblock Memory kernels from molecular dynamics.
\newblock {\em The Journal of Chemical Physics}, 75:2462, 1981.

\bibitem{beylkin2010approximation}
Gregory Beylkin and Lucas Monz{\'o}n.
\newblock Approximation by exponential sums revisited.
\newblock {\em Applied and Computational Harmonic Analysis}, 28(2):131--149,
  2010.

\bibitem{bleistein1975asymptotic}
Norman Bleistein and Richard~A Handelsman.
\newblock {\em Asymptotic expansions of integrals}.
\newblock Courier Corporation, 1975.

\bibitem{brenner2007mathematical}
S.~Brenner and R.~Scott.
\newblock {\em The Mathematical Theory of Finite Element Methods}.
\newblock Texts in Applied Mathematics. Springer New York, 2007.

\bibitem{LiXian2014}
Minxin Chen, Xiantao Li, and Chun Liu.
\newblock Computation of the memory functions in the generalized {Langevin}
  models for collective dynamics of macromolecules.
\newblock {\em J. Chem. Phys.}, 141:064112, 2014.

\bibitem{ChHaKu00}
A.~J. Chorin, O.~H. Hald, and R.~Kupferman.
\newblock Optimal prediction and the {Mori-Zwanzig} representation of
  irreversible processes.
\newblock {\em Proc. Nat. Acad. Sci. USA}, 97:6253 -- 6257, 2000.

\bibitem{ChHaKu02}
A.~J. Chorin, O.~H. Hald, and R.~Kupferman.
\newblock Optimal prediction with memory.
\newblock {\em Phys. D}, 166:239--257, 2002.

\bibitem{ChSt05}
A.~J. Chorin and P.~Stinis.
\newblock Problem reduction, renormalization, and memory.
\newblock {\em Comm. Appl. Math. Comp. Sc.}, 1:1--27, 2005.

\bibitem{Ciarlet1970}
Philippe~G. Ciarlet.
\newblock Discrete maximum principle for finite-difference operators.
\newblock {\em Aequationes mathematicae}, 4:338--352, 1970.

\bibitem{CuCe02}
S.~Curtaroo and G.~Ceder.
\newblock Dynamics of an inhomogeneously coarse grained multiscale system.
\newblock {\em Phys. Rev. Lett.}, 88:255504, 2002.

\bibitem{DobsonLuskin:2009b}
M.~Dobson and M.~Luskin.
\newblock An optimal order error analysis of the one-dimensional quasicontinuum
  approximation.
\newblock {\em SIAM J. Numer. Anal.}, 47:2455--2475, 2009.

\bibitem{weinan2007cauchy}
Weinan E and Pingbing Ming.
\newblock {Cauchy--Born} rule and the stability of crystalline solids: static
  problems.
\newblock {\em Archive for Rational Mechanics and Analysis}, 183(2):241--297,
  2007.

\bibitem{fricks2009time}
John Fricks, Lingxing Yao, Timothy~C Elston, and M~Gregory Forest.
\newblock Time-domain methods for diffusive transport in soft matter.
\newblock {\em SIAM journal on applied mathematics}, 69(5):1277--1308, 2009.

\bibitem{FuTh78}
E.R. Fuller and R.M. Thomson.
\newblock Lattice theories of fracture.
\newblock In {\em Fracture Mechanics of Ceramics; Proceedings of the
  International Symposium, University Park, PA.}, pages 507--548, 1978.

\bibitem{matrixlemma}
David~A. Harville.
\newblock {\em Matrix algebra from a statistician's perspective}.
\newblock Springer, New York, 1997.

\bibitem{hijon2010mori}
Carmen Hij{\'o}n, Pep Espa{\~n}ol, Eric Vanden-Eijnden, and Rafael
  Delgado-Buscalioni.
\newblock {Mori-Zwanzig} formalism as a practical computational tool.
\newblock {\em Faraday discussions}, 144:301--322, 2010.

\bibitem{hijon2006markovian}
Carmen Hij{\'o}n, Mar Serrano, and Pep Espa{\~n}ol.
\newblock {Markovian} approximation in a coarse-grained description of atomic
  systems.
\newblock {\em The Journal of chemical physics}, 125:204101, 2006.

\bibitem{horn2012matrix}
R.A. Horn and C.R. Johnson.
\newblock {\em Matrix Analysis}.
\newblock Cambridge University Press, 2012.

\bibitem{interlacing}
Suk-Geun Hwang.
\newblock Cauchy's interlace theorem for eigenvalues of {Hermitian} matrices.
\newblock {\em The American Mathematical Monthly}, 111(2):157--159, 2004.

\bibitem{kauzlaric2011bottom}
David Kauzlari{\'c}, Julia~T Meier, Pep Espa{\~n}ol, Sauro Succi, Andreas
  Greiner, and Jan~G Korvink.
\newblock Bottom-up coarse-graining of a simple graphene model: The blob
  picture.
\newblock {\em The Journal of chemical physics}, 134(6):064106--064106, 2011.

\bibitem{Kou_Xie_PRL_2004}
S.~C. Kou and X.~S. Xie.
\newblock Generalized {Langevin} equation with fractional {Gaussian} noise:
  Subdiffusion within a single protein molecule.
\newblock {\em Phys. Rev. Lett.}, 93:180603, 2004.

\bibitem{Kubo66}
R.~Kubo.
\newblock The fluctuation-dissipation theorem.
\newblock {\em Reports on Progress in Physics}, 29(1):255 -- 284, 1966.

\bibitem{lange2006collective}
Oliver~F Lange and Helmut Grubm{\"u}ller.
\newblock Collective {Langevin} dynamics of conformational motions in proteins.
\newblock {\em The Journal of chemical physics}, 124:214903, 2006.

\bibitem{Leach01}
A.R. Leach.
\newblock {\em Molecular modelling: principles and applications}.
\newblock Prentice Hall, 2001.

\bibitem{lee2017semi}
Chung-Shuo Lee, Yan-Yu Chen, Chi-Hua Yu, Yu-Chuan Hsu, and Chuin-Shan Chen.
\newblock Semi-analytical solution for the generalized absorbing boundary
  condition in molecular dynamics simulations.
\newblock {\em Computational Mechanics}, 60(1):23--37, 2017.

\bibitem{lei2016data}
Huan Lei, Nathan~A Baker, and Xiantao Li.
\newblock Data-driven parameterization of the generalized {Langevin} equation.
\newblock {\em Proceedings of the National Academy of Sciences},
  113(50):14183--14188, 2016.

\bibitem{lepri1998anomalous}
Stefano Lepri, Roberto Livi, and Antonio Politi.
\newblock On the anomalous thermal conductivity of one-dimensional lattices.
\newblock {\em EPL (Europhysics Letters)}, 43(3):271, 1998.

\bibitem{lepri2003thermal}
Stefano Lepri, Roberto Livi, and Antonio Politi.
\newblock Thermal conduction in classical low-dimensional lattices.
\newblock {\em Physics Reports}, 377(1):1--80, 2003.

\bibitem{Li2009c}
X.~Li.
\newblock A coarse-grained molecular dynamics model for crystalline solids.
\newblock {\em Int. J. Numer. Meth. Engng.}, 83:986--997, 2010.

\bibitem{Li14}
X.~Li.
\newblock Coarse-graining molecular dynamics models using an extended
  {Galerkin} projection.
\newblock {\em Int. J. Numer. Meth. Engng.}, 99:157--182, 2014.

\bibitem{LiE06}
X.~Li and W.~E.
\newblock Variational boundary conditions for molecular dynamics simulations of
  solids at low temperature.
\newblock {\em Comm. Comp. Phys.}, 1:136 -- 176, 2006.

\bibitem{LiE07}
X.~Li and W.~E.
\newblock Boundary conditions for molecular dynamics simulations at finite
  temperature: Treatment of the heat bath.
\newblock {\em Phys. Rev. B}, 76:104107, 2007.

\bibitem{LiMing:2013}
X.~Li and P.B. Ming.
\newblock On the effect of ghost force in the quasicontinuum method: dynamic
  problems in one dimension.
\newblock {\em Commun. Comput. Phys.}, 15:647--676, 2014.

\bibitem{LU2002119}
Tzon-Tzer Lu and Sheng-Hua Shiou.
\newblock Inverses of 2$\times$2 block matrices.
\newblock {\em Computers and Mathematics with Applications}, 43(1):119 -- 129,
  2002.

\bibitem{Ma2005}
Jianpeng Ma.
\newblock {Usefulness and limitations of normal mode analysis in modeling
  dynamics of biomolecular complexes.}
\newblock {\em Structure (London, England : 1993)}, 13(3):373--80, March 2005.

\bibitem{min2005observation}
Wei Min, Guobin Luo, Binny~J Cherayil, Samuel~C Kou, and X~Sunney Xie.
\newblock Observation of a power law memory kernel for distance fluctuation
  within a single protein molecule.
\newblock In {\em SPIE Third International Symposium on Fluctuations and
  Noise}, pages 194--204. International Society for Optics and Photonics, 2005.

\bibitem{MingYang:2009}
P.B. Ming and J.Z. Yang.
\newblock Analysis of a one-dimensional nonlocal quasi-continuum method.
\newblock {\em Multiscale Model. Simul.}, 7:1838--1875, 2009.

\bibitem{Mori65}
H.~Mori.
\newblock Transport, collective motion, and {Brownian} motion.
\newblock {\em Prog. Theor. Phys.}, 33:423 -- 450, 1965.

\bibitem{oliva2000generalized}
Baldomero Oliva, Xavier Daura, Enrique Querol, Francesc~X Avil{\'e}s, and
  O~Tapia.
\newblock A generalized {Langevin} dynamics approach to model solvent dynamics
  effects on proteins via a solvent-accessible surface. the carboxypeptidase a
  inhibitor protein as a model.
\newblock {\em Theoretical Chemistry Accounts}, 105(2):101--109, 2000.

\bibitem{stepanova2007dynamics}
Maria Stepanova.
\newblock Dynamics of essential collective motions in proteins: Theory.
\newblock {\em Physical Review E}, 76(5):051918, 2007.

\bibitem{zhu2017rigorous}
Yuanran Zhu, Jason~M Dominy, and Daniele Venturi.
\newblock Rigorous error estimates for the memory integral in the
  {Mori-Zwanzig} formulation.
\newblock {\em arXiv preprint arXiv:1708.02235}, 2017.

\bibitem{Zwanzig73}
R.~Zwanzig.
\newblock Nonlinear generalized {Langevin} equations.
\newblock {\em J. Stat. Phys.}, 9:215 -- 220, 1973.

\end{thebibliography}


\end{document}